\documentclass[12pt]{amsart}  
\usepackage{amsmath,amsthm,amssymb,amsfonts,eucal, latexsym,mathrsfs, stmaryrd}
\usepackage{comment}
\usepackage{color}
\usepackage{graphicx}
\usepackage{enumerate}
\usepackage{parskip}
\usepackage{mathdots}

\numberwithin{equation}{section}
\theoremstyle{plain}
\newtheorem{Proposition}[equation]{Proposition}

\newtheorem*{Corollary*}{Corollary}
\newtheorem{Theorem}[equation]{Theorem}
\newtheorem*{Theorem*}{Theorem}
\newtheorem{Lemma}[equation]{Lemma}
\theoremstyle{definition}
\newtheorem{Definition}[equation]{Definition}

\newtheorem{Example}[equation]{Example}
\newtheorem{Remark}[equation]{Remark}

\allowdisplaybreaks

\def\C{\mathbb{C}}
\def\R{\mathbb{R}}
\def\D{\mathbb{D}}
\def\T{\mathbb{T}}
\def\N{\mathbb{N}}

\def\phi{\varphi}
\def\M{\mathscr{M}}

\renewcommand{\leq}{\leqslant}
\renewcommand{\geq}{\geqslant}
\renewcommand{\subset}{\subseteq}

\begin{document}
%
%
%%%%%%%%%%%%%%%%%%%%%%%%%%%%%%%%%%%%%%%%%%%%%%%%%%%%%%%%%%%%%%%%%%%%%
%
%
%
%
%
%
%
%
%%%%%%%%%%%%%%%%%%%%%%%%%%%%%%%%%%%%%%%%%%%%%%%%%%%%%%%%%%%%%%%%%%%%%

\title{Inner functions and zero sets for $\ell^{p}_{A}$}

\author[Cheng]{Raymond Cheng}
\address{Department of Mathematics and Statistics,
  Old Dominion University,
  Norfolk, VA 23529}
  \email{rcheng@odu.edu}

\author[Mashreghi]{Javad Mashreghi}
\address{D\'epartement de math\'ematiques et de statistique, Universit\'e laval, Qu\'ebec, QC, Canada, G1V 0A6}
\email{javad.mashreghi@mat.ulaval.ca}

\author[Ross]{William T. Ross}
	\address{Department of Mathematics and Computer Science, University of Richmond, Richmond, VA 23173, USA}
	\email{wross@richmond.edu}

\begin{abstract}
In this paper we characterize the zero sets of functions from $\ell^{p}_{A}$ (the analytic functions on the open unit disk $\D$ whose Taylor coefficients form an $\ell^p$ sequence) by developing a concept of an ``inner function'' modeled by Beurling's discussion of the Hilbert space $\ell^{2}_{A}$, the classical Hardy space.  The zero set criterion is used to construct families of zero sets which are not covered by classical results.  In particular, we give an alternative proof of a result of Vinogradov \cite{Vinogradov} which says that when $p > 2$, there are zero sets for $\ell^{p}_{A}$ which are not Blaschke sequences.  
 \end{abstract}

%\subjclass[2010]{Primary: , Secondary: }

%\keywords{}

\thanks{This work was supported by NSERC (Canada).}

\maketitle

%%%%%%%%%%%%%%%%%%%%%%%%%%%%%%%%%%%%%%%%%%%%%%%%%%%%%%%%%%%%%%%%%%%%%
%%%%%%%%%%%%%%%%%%%%%%%%%%%%%%%%%%%%%%%%%%%%%%%%%%%%%%%%%%%%%%%%%%%%%
%%%%%%%%%%%%%%%%%%%%%%%%%%%%%%%%%%%%%%%%%%%%%%%%%%%%%%%%%%%%%%%%%%%%%

\section{Introduction}

For $p  \in (1, \infty)$ let 
$$\ell^{p}_{A} := \left\{f(z) = \sum_{k \geq 0} a_k z^k: \ \sum_{k \geq 0} |a_k|^p < \infty\right\}.$$
When endowed with the norm 
$$\|f\|_{p} := \left(\sum_{k \geq 0} |a_k|^p\right)^{1/p},$$
$\ell^{p}_{A}$ is a Banach space of analytic functions on the open unit disk $\D = \{z: |z| < 1\}$. Notice how we are identifying, via Taylor coefficients, i.e., 
$$\sum_{k \geq 0} a_k z^k \leftrightarrow (a_{k})_{k \geq 0},$$ the function space $\ell^{p}_{A}$ with the classical sequence space 
$$\ell^{p} = \left\{(a_k)_{k \geq 0}: \sum_{k \geq 0} |a_k|^p < \infty\right\}.$$ The Hausdorff-Young inequalities show that 
\begin{equation}\label{HYE}
 \ell^{p}_{A} \subset H^{p'}, \quad p \in (1, 2]; \qquad H^{p'} \subset \ell^{p}_{A}, \quad p \in [2, \infty).
\end{equation}
%when $p \in (1, 2]$, $\ell^{p}_{A} \subset H^{p'}$ and when $p \in [2, \infty)$, $H^{p'} \subset \ell^{p}_{A}$. 
Here $p'$ denotes the H\"{o}lder conjugate index to $p$, i.e., $1/p + 1/p' = 1$, and for $s \in (0, \infty)$, $H^{s}$ denotes the standard Hardy space of analytic functions $f$ on $\D$ for which 
$$\sup_{0 < r < 1} \int_{0}^{2 \pi}  |f(r e^{i \theta})|^{s} \frac{d \theta}{2 \pi} < \infty.$$  For more on the Hardy spaces, including the above-mentioned Hausdorff-Young inequalities, see \cite{Dur, MR2500010}. 

The Hardy spaces are well-understood spaces of analytic functions. Indeed, their zero sets, their boundary behavior, their multipliers, and their invariant subspaces under the shift operator $S f = z f$ are well-known and considered part of the classical complex analysis literature.  Similar topics have been explored for other related Banach spaces of analytic functions such as the Bergman spaces \cite{Duren-Bergman} and the Dirichlet spaces \cite{D-book}.   By comparison, relatively little is known about $\ell^{p}_{A}$ spaces \cite{AB, MR3714456} when $p > 1$. The Wiener algebra $\ell^{1}_{A}$ is somewhat better understood but there is still much work to be done \cite{MR973315}. In this paper we focus on the zero sets for $\ell^{p}_{A}$.

To speak of the ``zero set'' of a nontrivial, i.e., not identically zero, analytic function in $\mathbb{D}$, it is necessary to take account the multiplicities of the zeros.  With that in mind, we could consider systems
\[
     W = \{(s_1, s_2, s_3,\ldots), (n_1, n_2, n_3,\ldots)\},
\] 
where $s_1, s_2, s_3,\ldots$ are {\em distinct} points of $\mathbb{D}$, and $n_1, n_2, n_3,\ldots$ are positive integers.  For each $k\geq 1$, we would think of $s_k$ as a zero of multiplicity $n_k$.   
Equivalently, as we will do in this paper, we could work with finite or infinite sequences $W = (w_1, w_2, w_3,\ldots)$ of points in $\mathbb{D}$, with entries repeated in accordance with the multiplicity of the zero $w_k$.  In either case, let us define 
\begin{equation}\label{RW}
     \mathscr{R}_W := \{f \in \ell^{p}_{A}: f(w_k) = 0, \; \;  k \geq 1\},
\end{equation}
where in the above, if the zero $w_k$ is repeated, then we include the requirement that the requisite number of derivatives of $f$ vanish at $w_k$.  Observe that an $f \in \mathscr{R}_{W}$ might have other zeros besides $w_k$, or perhaps zeros of higher multiplicity than the multiplicity of $w_k$. Since convergence of a sequence of functions in the norm of $\ell^{p}_{A}$ implies uniform convergence on compact subsets of $\D$, via  \eqref{oeiwjfdkv} below, $\mathscr{R}_W$ is a closed subspace of $\ell^p_A$.

\begin{Definition} We say that a sequence  $W = (w_{k})_{k \geq 1} \subset \D$ is a {\em zero set} for $\ell^{p}_{A}$ if $\mathscr{R}_W \not = \{0\}$.
\end{Definition}

When $p = 2$, notice, via Parseval's theorem and radial boundary values, that $\ell^{2}_{A}$ is the classical Hardy space $H^2$. Furthermore, the zero sets for $\ell^{2}_{A} = H^2$ are well understood. 
Indeed, a sequence $(w_k)_{k \geq 1} \subset \mathbb{D} \setminus \{0\}$ is a zero set for $\ell^{2}_{A}$ if and only if 
$$\prod_{k \geq 1} \frac{1}{|w_k|} < \infty,$$  or equivalently,
\begin{equation}\label{024ouriehtrjgeqw}
\sum_{k \geq 1} (1 - |w_k|) < \infty.
\end{equation}
This last summability condition is known as the {\em Blaschke condition} and
such sequences $(w_k)_{k \geq 1}$ are known as {\em Blaschke sequences}.  One example of an $f \in \ell^{2}_{A} \setminus \{0\}$ vanishing on a Blaschke sequence $(w_k)_{k \geq 1}$, along with the desired multiplicities, is the Blaschke product 
\begin{equation}\label{wieurhgkfjldbg}
f(z) = \prod_{k \geq 1} \frac{|w_k|}{w_k} \frac{w_k - z}{1 - \overline{w_k} z}.
\end{equation}
Note that this Blaschke product is a bounded analytic function on $\D$.

Though the description of the zero sets for $\ell^{p}_{A}$ for $p \not = 2$ does not seem to have such a nice summability condition as with $p = 2$, there are a few things we can say about the zero sets for $\ell_{A}^{p}$. When $p \in (1, 2]$, the Hausdorff-Young inequalities from \eqref{HYE}  show that $\ell^{p}_{A} \subset H^{p'}$. Consequently, since the zeros of functions in the Hardy spaces must be Blaschke sequences, the zero set for any $f\in\ell^p_A \setminus \{0\}$ must be a Blaschke sequence. But this is only a necessary condition. 
For an example of a sufficient condition, a result of Newman and Shapiro \cite{MR0148874} shows that if the sequence $(w_k)_{k \geq 1} \subset \D$ converges exponentially to the unit circle $\T$, i.e., 
\begin{equation}\label{pppsdofposfpo11}
1 - |w_{k + 1}| < c (1 - |w_{k}|)
\end{equation} for some $c \in (0, 1)$, then the Blaschke product $B$ with precisely these zeros has Fourier coefficients $\widehat{B}(n)$ which satisfy $|\widehat{B}(n)| = O(1/n)$ and so $B \in \ell^{p}_{A}$. It is also known \cite[Ch.~III]{MR947146} that if the sequence $(w_k)_{k \geq 1}$ accumulates on a sufficiently thin subset of $\T$, then this sequence is the zero set of an analytic function on $\D$ which has an infinitely differentiable continuation to the closure of $\D$. Such functions have rapidly decaying Taylor coefficients and thus belong to $\ell^{p}_{A}$. Thus these types of ``thin'' sequences are also zero sets for $\ell^{p}_{A}$.  In the last several sections of this paper, we will construct further examples of zero sets of $\ell^p_A$. One of the results in this regard is the following.  Below,  $r'$ denotes the H\"{o}lder conjugate index of $r$.

\begin{Theorem}\label{blaslikeex}
   Fix $p \in (1, \infty)$ and let $W = (w_1, w_2, w_3,\ldots) \subset \mathbb{D} \setminus \{0\}$.  
Choose numbers $r_1, r_2, r_3,\ldots$ with  $r_k >1$ and 
\[
       \sum_{k \geq 1} \Big(1 - \frac{1}{r_k}\Big) < \frac{1}{p'}.
\]
If the sequence $W$ satisfies
$$
     \sum_{k \geq 1}  (1 - |w_k|^{r'_{k}})^{r_k-1} < \infty,
$$
then $W$ is a zero set for $\ell^p_A$.
\end{Theorem}

It is still an open question as to whether {\em every} Blaschke sequence is a zero set for $\ell^{p}_{A}$. 

When $p \in (2, \infty)$, the Hausdorff-Young inequalities from \eqref{HYE} yield the containment $H^{p'} \subset \ell^{p}_{A}$, which means that any Blaschke sequence is a zero set for $\ell^{p}_{A}$. Furthermore, a result of Vinogradov \cite{Vinogradov} shows that when $p > 2$ the zero sequences for $(w_k)_{k \geq 1}$ must satisfy the condition 
$$\sum_{k \geq 1} (1 - |w_k|)^{1 + \epsilon} < \infty$$
for any $\epsilon > 0$, which is somewhat weaker than the Blaschke condition. 

On the other hand, is it possible for certain non-Blaschke sequences to be zero sets for $\ell^{p}_{A}$? One might suspect the answer is yes due to the fact that when $p > 2$, functions in $\ell^{p}_{A}$  can be poorly behaved near $\T$ \cite[Cor.~5.2]{MR3714456} in that there are functions in $\ell^{p}_{A}$ which do not possess radial limits almost everywhere on $\T$.  This is contrast to the Hardy space case where there is an expected radial, even non-tangential, regularity near $\T$. All this leads to the suspicion that some of the zero sets for $\ell^{p}_{A}$, $p > 2$, might be non-Blaschke.  As shown by Vinogradov \cite{Vinogradov}, this is indeed the case. As a consequence of the main result of this paper (see Theorem \ref{italiansubthm} below), using a a notion of ``inner function'' for $\ell^{p}_{A}$, we will give a new proof of this fact. 

\begin{Theorem}\label{nbcxuywrgth}
For each $p > 2$, there exists a non-Blaschke sequence $(w_k)_{k \geq 1} \subset \D$, i.e., 
$$\sum_{k \geq 1} (1 - |w_k|) = \infty$$ that is a zero sequence for $\ell^{p}_{A}$. 
\end{Theorem}

The zero sequence created in the theorem above will be wildly distributed near the circle $\T$ since, as is well known by a result of Shapiro and Shields \cite{MR0145082}, that if the zeros of a nontrivial  $\ell^{p}_{A}$ function all lie on a single radius, or more generally in a single non-tangential approach region, then they must form a Blaschke sequence. Also noteworthy here is the fact that when $p > 2$, $\ell^{p}_{A}$ is considered a ``large'' space of analytic functions on $\D$. Another ``large'' space, the Bergman space $A^{p}$, also has the property that their zero sets need not always be Blaschke sequences \cite{Duren-Bergman}. We also point out here that the function from Theorem \ref{nbcxuywrgth} vanishing on this non-Blaschke sequence will be an ``inner function'' for $\ell^{p}_{A}$. 

\noindent {\bf Acknowledgement:} We wish to thank the referee for their careful reading of our paper and their useful comments.

\section{Our approach}

The main tools we use to explore the zero sets for $\ell^{p}_{A}$ are geometric ideas, including a notion of orthogonality on normed linear spaces developed by Birkhoff and James \cite{AMW, Jam}. This notion of orthogonality was also used  in some recent papers \cite{CR, CR2, MR3686895}.  We use this idea to define a sort of ``inner function'' for $\ell^p_A$  following an approach first exploited by Beurling when examining the shift invariant subspaces of  the Hardy space $\ell^{2}_{A}$.

With our notation and background to be explained in greater detail in the next two sections, we will say that $J \in \ell^{p}_{A} \setminus \{0\}$ is a {\em $p$-inner function} if 
$$J \perp_{p} S^{n} J, \quad n \in \N,$$
where $\perp_p$ is the notion of orthogonality on $\ell^{p}_{A}$ developed by Birkhoff and James and will be explained below in \eqref{2837eiywufh[wpofjk}.
As usual, 
$$S: \ell^{p}_{A} \to \ell^{p}_{A}, \quad (S f)(z) = z f(z),$$ is the unilateral shift operator, which is an isometry on $\ell^{p}_{A}$.  
When $p = 2$, we can use Parseval's theorem to see that $J \in \ell^{2}_{A} \setminus \{0\}$ is $2$-inner when the radial boundary function for $J$, i.e., 
\begin{equation}\label{radial_BF}
J(e^{i \theta}) := \lim_{r \to 1^{-}} J(r e^{i \theta}),
\end{equation}
which exists almost everywhere and defines a function in $L^2(d \theta/2 \pi)$ \cite{Dur}, satisfies the integral condition 
$$\int_{0}^{2 \pi} |J(e^{i \theta})|^2 e^{-i n \theta} \frac{d \theta}{2 \pi} = 0, \quad n \in \N.$$
The identity in the previous line, along with the same identity but taking complex conjugates, shows that a nontrivial function $J$ is $2$-inner precisely when the radial boundary function for $J$ has constant, and nonzero, modulus almost everywhere on the unit circle $\T$; that is, apart from a multiplicative constant, $J$ is inner in the traditional sense.

For $f \in \ell^{p}_{A} \setminus \{0\}$, let $\widehat{f}$ be the metric projection of $f$ onto the subspace
$$[S f] := \bigvee\{ S^{k} f: k \geq 1\},$$
i.e., the closest point in $[S f]$ to $f$ (which exists and is unique since $\ell^{p}_{A}$ is a uniformly convex Banach space \cite{Clark}). Our first observation is that 
$$J = f - \widehat{f}$$ defines a $p$-inner function and indeed all $p$-inner functions arise in this manner (see Proposition \ref{poiutorgorrrefdxxx} below). Note that when $p = 2$, $\ell^{2}_{A}$ is a Hilbert space and the above metric projection $\widehat{f}$ is the orthogonal projection of $f$ onto $[S f]$. Beurling used this co-projection analysis to describe the $S$-invariant spaces of $\ell^{2}_{A}$. It is worth pointing out that the invariant subspaces of $\ell^{p}_{A}$ when $p \not = 2$ can be quite complicated. Indeed, as was shown in \cite{AB}, the $\ell^{p}_{A}$ spaces, for $p > 2$, have the so-called {\em co-dimension $n$ property} in that given $n \in \N \cup \{\infty\}$ there is an $S$-invariant subspace $\M$ of $\ell^{p}_{A}$ for which $\operatorname{dim}(\M/S \M) = n$. This is in sharp contrast to Beurling's theorem which says that $\operatorname{dim}(\M/S\M) = 1$ for any non-trivial $S$-invariant subspace $\M$ of $\ell^{2}_{A} = H^2$. This co-dimension $n$ phenomenon appears in the Bergman space setting as well \cite{Duren-Bergman}.

We  mention that an analogous notion of ``inner'' for the Bergman space, along with some associated extremal conditions, was explored in \cite{MR1440934, MR1278431, MR1398090, MR1197044}, while the corresponding notions for the Dirichlet space was explored in \cite{MR936999, Seco}. In these settings there is also a notion of ``outer'' along with a corresponding ``inner-outer factorization'' which mirrors the classical Nevanlinna factorization in the Hardy spaces \cite{Dur}. To a somewhat weaker sense, there is a notion of inner-outer factorization in $\ell^{p}_{A}$ \cite{CR2}.

We can now state the main theorem of the paper which uses our concept of $p$-inner functions to describe the zero sets for $\ell^{p}_{A}$.

\begin{Theorem}\label{italiansubthm}
   Let $p \in (1, \infty)$ and suppose that $W = (w_1, w_2,\ldots) \subset \mathbb{D}\setminus \{0\}$.
\  Define, for each $n \in \N$,
 \begin{equation}\label{9384yrteiugfdsa}
        f_n(z) := \Big(1 - \frac{z}{w_1}  \Big)\Big(1 - \frac{z}{w_2}  \Big)\cdots \Big(1 - \frac{z}{w_n}  \Big)
 \end{equation}
and 
$$J_n := f_n - \widehat{f_n}.$$
Then 
\begin{enumerate}
\item $J_n$ is $p$-inner;
\item $\|J_n\|_p$ is monotone increasing with $n$; 
\item $W$ is a zero set for $\ell^{p}_{A}$ if and only if 
$$\sup_{n \geq 1} \|J_{n}\|_{p} < \infty.$$
In this case,
$J_{n}$ converges in the norm of $\ell^{p}_{A}$ to a $p$-inner function $J \in \ell^{p}_{A} \setminus \{0\}$ such that $J\in\mathscr{R}_W$.
\end{enumerate}
\end{Theorem}

When $p = 2$, the $2$-inner functions $J_n$ corresponding to $f_n$ from \eqref{9384yrteiugfdsa} turn out to be 
$$J_{n}(z) = \left(\prod_{k = 1}^{n} \frac{1}{w_k}\right) \prod_{k = 1}^{n} \frac{w_k - z}{1 - \overline{w_k} z},$$
which, apart from a multiplicative constant,  are finite Blaschke products.  Furthermore, Parseval's theorem, and the fact that 
$$\Big|\prod_{k = 1}^{n} \frac{w_k - e^{i \theta}}{1 - \overline{w_k} e^{i \theta}}\Big| = 1, \quad \theta \in [0, 2 \pi],$$
shows that 
$$\|J_n\|_{2} = \prod_{k = 1}^{n} \frac{1}{|w_k|}.$$ Thus in the $p = 2$ case, the necessary and sufficient condition for $W = (w_k)_{k \geq 1} \subset \D \setminus \{0\}$ to be a zero set for $\ell^{2}_{A} = H^2$ is 
$$\sup_{n \geq 1} \|J_{n}\|_{2} = \sup_{n \geq 1} \prod_{k = 1}^{n} \frac{1}{|w_k|} < \infty,$$ which, as noted already in \eqref{024ouriehtrjgeqw}, is precisely the well-known Blaschke condition. 

In final sections of this paper we will use Theorem \ref{italiansubthm} along with estimates for the norms $\| J_n \|_p$   to derive sufficient conditions for zero sets for $\ell^{p}_{A}$ spaces.  

We will also derive formulas for $p$-inner functions corresponding to finite zero sets.  Related extremal problems will enable us to control the limiting case of infinite zero sets.   Theorem \ref{italiansubthm} also enables us to construct new concrete examples of zero sets for $\ell^{p}_{A}$, going beyond those previously known from \cite{MR0148874,MR947146}.
As mentioned earlier, they include the important new example of non-Blaschke zero sequences for $\ell^{p}_{A}$ when $p>2$.

%%%%%%%%%%%%%%%%%%%%%%%%%%%%%%%%%%%%%%%%%%%%%%%%%%%%%%%%%%%%%%%%%%%%%
%%%%%%%%%%%%%%%%%%%%%%%%%%%%%%%%%%%%%%%%%%%%%%%%%%%%%%%%%%%%%%%%%%%%%
%%%%%%%%%%%%%%%%%%%%%%%%%%%%%%%%%%%%%%%%%%%%%%%%%%%%%%%%%%%%%%%%%%%%%

\section{The basics of $\ell^{p}_{A}$ spaces}

For $p \in (1, \infty)$ let $\ell^p$ denote the set of sequences
$$\mathbf{a} = (a_0, a_1, a_2,\ldots)$$ of complex numbers for which
$$\|\mathbf{a}\|_{p} := \left(\sum_{k  \geq 0} |a_k|^{p}\right)^{1/p} < \infty.$$
The quantity $\|\mathbf{a}\|_{p}$ defines a norm which makes $\ell^{p}$ a uniformly convex Banach space \cite[p.~117]{Car}. 

For any $p \in (1, \infty)$, let $p'$ denote the usual H\"{o}lder conjugate index to $p$. We know that $(\ell^{p})^{*}$, the norm dual of $\ell^{p}$, can be isometrically  identified with
$\ell^{p'}$ by means of the bi-linear pairing
\begin{equation}\label{BLP}
(\mathbf{a}, \mathbf{b}) := \sum_{k \geq 0} a_k b_k, \quad \mathbf{a} \in \ell^{p}, \mathbf{b} \in \ell^{p'}.
\end{equation}

We now equate $\ell^p$ with a space of analytic functions on $\D$ as follows. For an $\mathbf{a} \in \ell^p$ we set
\begin{equation}\label{bvfhjugi9r}
a(z) := \sum_{k \geq 0} a_k z^k
\end{equation}
 to be the power series whose sequence of Taylor coefficients is equal to $\mathbf{a}$. Note the use of $\mathbf{a}$ (bold faced) to represent a sequence and $a$ (not bold faced) to represent the corresponding power series. By H\"{o}lder's inequality we see that
\begin{align*}
\sum_{k \geq 0} |a_k| |z^k| & \leq \left(\sum_{k \geq 0} |a_k|^p\right)^{1/p} \left(\sum_{k \geq 0} |z|^{k p'}\right)^{1/p'}\\
& = \|\mathbf{a}\|_{p} \left( \frac{1}{1 - |z|^{p'}}\right)^{1/p'}, \quad z \in \D.
\end{align*}
 This implies that the above power series in \eqref{bvfhjugi9r} converges uniformly on compact subsets of $\D$ and thus determines an analytic function on $\mathbb{D}$. If we define
$$\ell^{p}_{A} := \{a: \mathbf{a} \in \ell^p\}$$ and endow $a$ with the norm $\|\mathbf{a}\|_{p}$, i.e., 
$$\|a\|_{p} := \left(\sum_{k \geq 0} |a_k|^p \right)^{1/p},$$
then $\ell^p_{A}$ becomes a Banach space of analytic functions on $\mathbb{D}$. Furthermore, for each $z \in \mathbb{D}$ and $a \in \ell^{p}_{A}$ we have
\begin{equation}\label{oeiwjfdkv}
|a(z)| \leq \|a\|_{p} \left( \frac{1}{1 - |z|^{p'}}\right)^{1/p'}.
\end{equation}
Thus if a sequence of functions from $\ell^{p}_{A}$ converges in the norm of $\ell^{p}_{A}$, then it converges uniformly on compact subsets of $\mathbb{D}$. Also worth pointing out here is that \eqref{oeiwjfdkv} extends to derivatives in that if $n \in \N \cup \{0\}$ and $K$ is a compact subset of $\D$, then 
\begin{equation}\label{4398742eqw}
|a^{(n)}(z)| \leq C_{n, K} \|a\|_{p}, \quad z \in K.
\end{equation}
One can prove this by using \eqref{oeiwjfdkv} and the Cauchy integral formula. 

\begin{Remark}
  We emphasize that when $p \neq 2$, $\|f\|_p$ is the norm in $\ell^{p}_{A}$ of its coefficient sequence, and {\em not}, as the notation might appear to suggest, its norm in any Hardy space $H^p$.
\end{Remark}

When $p \in (1, \infty)$, the space $\ell^{p}$, and hence $\ell^{p}_{A}$, is uniformly convex. Consequently it enjoys the unique {\em nearest-point property}, i.e., for each closed subspace $\mathcal{M} \subset \ell^{p}_{A}$ and any $a \in \ell^{p}_{A}$, there is a unique $a_{\mathcal{M}} \in \mathcal{M}$ satisfying 
$$\|a - a_{\mathcal{M}}\|_{p} = \inf\{\|a - b\|_{p}: b \in \mathcal{M}\}.$$

\begin{Definition}\label{24387tr9euoifjdkw}
The vector $a_{\mathcal{M}}$ is called the {\em metric projection} of $a$ onto $\mathcal{M}$ while the vector $a - a_{\mathcal{M}}$ is called the {\em metric co-projection}. 
\end{Definition}

It is important to note here that despite the use of the suggestive word ``projection,'' the mapping $a \mapsto a_{\mathcal{M}}$ is generally nonlinear. Of course when $p = 2$, the map $a \mapsto a_{\mathcal{M}}$ is the orthogonal projection of $a$ onto $\mathcal{M}$. 

For fixed $w \in \mathbb{D}$ and an analytic function $f$ on $\mathbb{D}$, define the difference quotient 
\begin{equation}\label{Qw}
(Q_wf)(z) := \frac{f(z) - f(w)}{z - w} 
\end{equation}
and note that $Q_{w} f$ is analytic on $\D$. 
In many well-known Banach spaces of analytic functions on $\mathbb{D}$ (Hardy space, Bergman space, Dirichlet space) we can ``divide out a zero" and still remain in the space. Equivalently stated, if $f$ belongs to the space then so does $Q_wf$. The $\ell^{p}_{A}$ spaces also enjoy this property. 

\begin{Proposition}\label{jjjjdjdjsslpl}
Let $p \in (1, \infty)$ and $w \in \mathbb{D}$.  If $f \in \ell^p_A$  then $Q_{w} f \in \ell^{p}_{A}$. 
\end{Proposition}
\begin{proof}  Let
$
   f(z) = \sum_{k \geq 0} a_k z^k.
$ 
If $w = 0$, then the assertion is trivial.  Otherwise, using the identity 
$$\frac{f(z) - f(w)}{z - w} = \sum_{n \geq 0} \Big(\sum_{k \geq 0} a_{k + n + 1} w^k\Big) z^n,$$ 
it is elementary to show that 
 when $p=1$ we have
\[
\left\|   \frac{f(z) - f(w)}{z-w}  \right\|_1
    \leq  \| f \|_1 \cdot \frac{1}{1-|w|},
\]
and similarly for $p=\infty$,
\[
\left\|   \frac{f(z) - f(w)}{z-w}  \right\|_{\infty} 
    \leq  \| f \|_{\infty} \cdot \frac{1}{1-|w|},
\]
where in the above we are using the norm
$$\|a\|_{\infty} := \sup\{|a_k|: k \geq 0\}$$
on $\ell^{\infty}_{A}$ (analytic functions on $\D$ with bounded Taylor coefficients). 
The Riesz-Thorin interpolation theorem \cite{BL} shows that the linear operator $Q_w$ is bounded for any $p \in (1, \infty)$, with norm bounded above by $(1-|w|)^{-1}$ as well.    \end{proof}

%%%%%%%%%%%%%%%%%%%%%%%%%%%%%%%%%%%%%%%%%%%%%%%%%%%%%%%%%%%%%%%%%%%%%
%%%%%%%%%%%%%%%%%%%%%%%%%%%%%%%%%%%%%%%%%%%%%%%%%%%%%%%%%%%%%%%%%%%%%
%%%%%%%%%%%%%%%%%%%%%%%%%%%%%%%%%%%%%%%%%%%%%%%%%%%%%%%%%%%%%%%%%%%%%

\section{Banach Space Geometry}

The notion of Birkhoff-James orthogonality \cite{AMW, Jam} extends the concept of orthogonality from an inner product space to a more general normed linear space.   Let $\mathbf{x}$ and $\mathbf{y}$ be vectors from  a normed linear space $\mathcal{X}$.  We say that $\mathbf{x}$ is {\em orthogonal} to $\mathbf{y}$ in the Birkhoff-James sense if
\begin{equation}\label{2837eiywufh[wpofjk}
      \|  \mathbf{x} + \beta \mathbf{y} \|_{\mathcal{X}} \geq \|\mathbf{x}\|_{\mathcal{X}}
\end{equation}
for all scalars $\beta$. In this situation we write $\mathbf{x} \perp_{\mathcal{X}} \mathbf{y}$. The relation $\perp_{\mathcal{X}}$ is generally neither symmetric nor linear.    It is straightforward to show that when $\mathcal{X}$ is a Hilbert space, with inner product $\perp$, then $\mathbf{x} \perp \mathbf{y} \iff \mathbf{x} \perp_{\mathcal{X}} \mathbf{y}$. 

When $\mathcal{X} = \ell^{p}$,  let us write $\perp_p$ in place of the more cumbersome $\perp_{\ell^p}$.  Of particular importance here
is the following explicit criterion for   the relation $\perp_p$ when $p \in (1, \infty)$:
\begin{equation}\label{BJp}
 \mathbf{a} \perp_{p} \mathbf{b}  \iff   \sum_{k \geq 0} |a_k|^{p - 2} \overline{a_k} b_k  = 0,
\end{equation}
where any occurrence of ``$|0|^{p - 2} 0$'' in the sum above is interpreted as zero \cite[Example 8.1]{Jam}.

Borrowing from the expression in \eqref{BJp} we define, for a complex number $z=re^{i\theta}$, and any $s > 0$, the quantity
\begin{equation}\label{definition-de-zs}
z^{\langle s \rangle} = (re^{i\theta})^{\langle s \rangle} := r^{s} e^{-i\theta}.
\end{equation}
Let us begin by noting some simple properties of this non-linear operation, which is tied to Birkhoff-James orthogonality in $\ell^p$. We leave the verification to the reader.

\begin{Lemma}\label{776652bbbb}
 Let $p \in (1, \infty)$ and $p'$ denote the H\"{o}lder conjugate of $p$. Then for $w, z \in \mathbb{C}$, $n \in \mathbb{N}_{0}$, and $s > 0$, we have 
$$  (zw)^{\langle p-1\rangle} = z^{\langle p-1\rangle} w^{\langle p-1\rangle},  $$
      $$|z|^p =  z^{\langle p-1 \rangle}z,$$
      \[
      (z^{\langle s \rangle})^n = (z^n)^{\langle s \rangle},
      \]
      $$\quad (z^{\langle p-1 \rangle})^{\langle p'-1 \rangle} = z.
$$
\end{Lemma}

In light of the definition \eqref{definition-de-zs}, for $\mathbf{a} = (a_k)_{k \geq 0}$, let us write
 \begin{equation}\label{uewuweuer}
\mathbf{a}^{\langle p - 1 \rangle} := (a_k^{\langle p - 1 \rangle})_{k \geq 0}.
\end{equation}

If $\mathbf{a} \in \ell^{p}$, one can see that $\mathbf{a}^{\langle p - 1\rangle} \in \ell^{p'}$ \cite{CR2} and thus from \eqref{BJp},
\begin{equation}\label{ppsduwebxzrweq5}
\mathbf{a} \perp_{p} \mathbf{b} \iff (\mathbf{a}^{\langle p - 1\rangle}, \mathbf{b})= 0,
\end{equation}
where recall the bi-linear pairing $(\cdot, \cdot)$ from \eqref{BLP}.
Note that  in this special case $\perp_{p}$ is therefore linear in its second argument when $p \in (1, \infty)$, and thus it makes sense to speak of a vector being orthogonal to a subspace of $\ell^{p}$.  In fact, if $\mathcal{M}$ is a subspace of $\ell^{p}$, one can use \eqref{2837eiywufh[wpofjk} to see that for any vector $\mathbf{a} \in \ell^{p}$, the metric projection 
$\mathbf{a}_{\mathcal{M}}$ from Definition \ref{24387tr9euoifjdkw} satisfies 
\begin{equation}\label{poisdpfo876w8rg}
          \mathbf{a} - \mathbf{a}_{\mathcal{M}} \perp_{p} \mathcal{M}.
\end{equation}

With the isometric identification of $\ell^{p}_{A}$ with $\ell^{p}$ via Taylor coefficients, i.e., 
$$\mathbf{a}  = (a_{k})_{k \geqslant 0} \longleftrightarrow a(z) = \sum_{k \geq 0} a_{k} z^k,$$
we can pass the Birkhoff-James orthogonality from $\ell^{p}$ to $\ell^{p}_{A}$ via $$a(z) = \sum_{k \geq 0} a_k z^k, \quad b(z) = \sum_{k \geq 0} b_k z^k,$$
and write
       $$a \perp_p b \iff \mathbf{a} \perp_p \mathbf{b}.$$
       Similarly, we have the identification
$$ a^{\langle p - 1 \rangle}(z) = \sum_{k \geq 0} a_k^{\langle p - 1 \rangle} z^k.$$

%Another important tool from Banach space geometry is a sort of ``Pythagorean theorem'' with respect to Birkhoff-James orthogonality.  It takes the form of a family of inequalities. 

%\begin{Definition}\label{ndfgeirt}
%A Banach space  $\mathscr{X}$  satisfies a {\em Lower Pythagorean Inequality} if there exist  a $K>0$ and an $r \in (1,\infty)$ such that
%
%$$
%  \| x \|^r  +  K\|y\|^r \leq \|x + y \|^r
%$$
%whenever $x \perp_{\mathscr{X}} y$ in $\mathscr{X}$.
%
%Similarly,  $\mathscr{X}$  satisfies an {\em Upper Pythagorean Inequality} if there exist a  $K>0$ and an $r \in (1,\infty)$ such that
%
%$$
%  \| x \|^r  +  K\|y\|^r \geq \|x + y \|^r
%$$
%whenever $x \perp_{\mathscr{X}} y$ in $\mathscr{X}$.
%
%\end{Definition}
%

The proofs of our main results will need the following ``Pythagorean inequalities'' from \cite{CR}.

\begin{Theorem}\label{pythagthm}
Suppose that $\mathbf{x} \perp_p \mathbf{y}$ in $\ell^p$.
If $p \in (1, 2]$, then
\begin{align*}
   \| \mathbf{x} + \mathbf{y} \|^p_p & \leq  \|\mathbf{x}\|^p_p + \frac{1}{2^{p-1}-1}\|\mathbf{y}\|^p_p \\
    \| \mathbf{x} + \mathbf{y} \|^2_p & \geq  \|\mathbf{x}\|^2_p + (p-1)\|\mathbf{y}\|^2_p.
\end{align*}
If $p \in [2, \infty)$, then
\begin{align*}
   \| \mathbf{x} + \mathbf{y} \|^p_p & \geq  \|\mathbf{x}\|^p_p + \frac{1}{2^{p-1}-1}\|\mathbf{y}\|^p_p \\
    \| \mathbf{x} + \mathbf{y} \|^2_p & \leq  \|\mathbf{x}\|^2_p + (p-1)\|\mathbf{y}\|^2_p.
\end{align*}
\end{Theorem}
These inequalities have their origins in \cite{Byn, BD, CMP1} while a more general approach is presented in \cite{CH2,CR}.  When $p=2$, the four inequalities merely simplify to the familiar Pythagorean theorem for the Hilbert space $\ell^2$.

We will need the following metric projection lemma which is a consequence of these Pythagorean inequalities. 

%\begin{Lemma}\label{propnestlwplem}
%     Let $\mathscr{X}$ be a uniformly convex Banach space satisfying the lower Pythagorean inequality with constant $K$ and exponent $r$.  Suppose that $\mathscr{X}_n$ is a closed subspace of $\mathscr{X}$ for each $n \in \mathbb{N}$, such that
 %    \[
  %        \mathscr{X}_1 \supseteq \mathscr{X}_2 \supseteq \mathscr{X}_3 \supseteq \cdots.
 %    \]
 %    Define $\mathscr{X}_{\infty} = \bigcap_{n=1}^{\infty} \mathscr{X}_n$.   If $P_n$ is the metric projection mapping from $\mathscr{X}$ to $\mathscr{X}_n$, for all $n \in \mathbb{N} \cup \{\infty\}$, then for any $x \in \mathscr{X}$, $P_n x$ converges to $P_{\infty} x$ in norm.
%\end{Lemma}

\begin{Lemma}\label{propnestlwplem}
  Suppose that $\mathscr{X}_n$ is a closed subspace of $\ell^{p}$ for each $n \in \mathbb{N}$, such that
     \[
          \mathscr{X}_1 \supseteq \mathscr{X}_2 \supseteq \mathscr{X}_3 \supseteq \cdots.
     \]
     Define $\mathscr{X}_{\infty} = \bigcap_{n=1}^{\infty} \mathscr{X}_n$.   If, for each $n \in \N \cup \{\infty\}$, $P_n$ is the metric projection mapping from $\ell^p$ to $\mathscr{X}_n$, then for any $\mathbf{x} \in \mathscr{X}$, $P_n \mathbf{x}$ converges to $P_{\infty} \mathbf{x}$ in norm. % Note: we cannot say "strong operator topology" since these ain't operators!
\end{Lemma}

\begin{proof}  The space $\ell^{p}$ is uniformly convex, and hence has unique nearest points to subspaces, and by Theorem \ref{pythagthm} satisfies, for some $K, r > 0$,
\[
    \| \mathbf{x} + \mathbf{y} \|_{p}^r \geq \|\mathbf{x}\|_{p}^r + K \|\mathbf{y}\|_{p}^r
\]
whenever $\mathbf{x} \perp_{p} \mathbf{y}$.    Let $\mathbf{x} \in \ell^{p}$.  By the definition of metric projection from Definition \ref{24387tr9euoifjdkw}, whenever $m<n$, we have
\begin{align}
    \| \mathbf{x} - P_m \mathbf{x} \|_{p} &=  \inf\{ \|\mathbf{x}-\mathbf{z}\|_{p} : \mathbf{z} \in \mathscr{X}_m\} \nonumber \\
       &\leq \inf\{ \|\mathbf{x}-\mathbf{z}\|_{p} : \mathbf{z} \in \mathscr{X}_n\} \nonumber \\
       &= \|\mathbf{x} - P_n \mathbf{x}\|_{p} \nonumber \\
       &\leq \|\mathbf{x} - P_{\infty} \mathbf{x}\|_{p}.\label{83483495345}
\end{align}
Thus, as a sequence indexed by $n$, the quantity $\|\mathbf{x} - P_n \mathbf{x}\|_{p}$ is monotone nondecreasing, and bounded above.  Accordingly, it converges.  

Next, for $m<n$, the vector $P_m \mathbf{x} - P_n \mathbf{x}$ lies in $\mathscr{X}_m$ (the larger space), and hence by \eqref{poisdpfo876w8rg}, the co-projection $\mathbf{x} - P_m \mathbf{x}$ is Birkhoff-James orthogonal to it.
Consequently, again by Theorem \ref{pythagthm},
\[
      \| \mathbf{x} - P_n \mathbf{x} \|_{p}^r \geq \|\mathbf{x} - P_m \mathbf{x} \|_{p}^r + K\|P_m \mathbf{x} - P_n \mathbf{x}\|_{p}^r.
\]
Since the difference $\| \mathbf{x} - P_n \mathbf{x} \|_{p}^r  - \|\mathbf{x} - P_m \mathbf{x} \|_{p}^r$, which is non-negative, can be made arbitrarily small by choosing $m$ sufficiently large, it follows that $(P_m \mathbf{x})_{m \geq 1}$ is a Cauchy sequence in norm, and converges to some vector $\mathbf{z}$.  It is clear that $\mathbf{z}$ belongs to $\mathscr{X}_{\infty}$ since $(P_{m} \mathbf{x})_{m \geq N} \subset \mathscr{X}_{N}$ and thus $\mathbf{z} \in \mathscr{X}_{N}$ for all $N$.  From \eqref{83483495345} we have 
\[
      \|\mathbf{x}-\mathbf{z}\|_{p} \leq \|\mathbf{x} - P_{\infty} \mathbf{x}\|_{p},
\]
and, from the definition of $P_{\infty} \mathbf{x}$ being a metric projection, we see that equality must hold. Hence, by uniqueness (note another use of the uniform convexity of $\ell^{p}$ here), it must be that $\mathbf{z} = P_{\infty}\mathbf{x}$. \end{proof}

Birkhoff-James orthogonality arises in a natural way when considering extremal problems in Banach spaces.  In this setting, the Pythagorean inequalities are enormously useful for estimation.  These geometric ideas have been applied in estimating zeros of analytic functions \cite{MR3686895} and exploring a sort of ``inner-outer factorization'' for $\ell^p_A$ \cite{CR2}.
They have also proved fruitful in the study of stochastic processes endowed with an $L^p$ structure.
These processes include $\alpha$-stable processes with $\alpha \in (1, 2]$, $L^p$-harmonizable processes, and strictly stationary $L^p$ processes.  The orthogonality condition is connected to associated prediction problems, Wold-type decompositions, and moving-average representations \cite{CHW2, CH1,CMP,CMP1, CR,CR2,MP}.

%%%%%%%%%%%%%%%%%%%%%%%%%%%%%%%%%%%%%%%%%%%%%%%%%%%%%%%%%%%%%%%%%%%%%
%%%%%%%%%%%%%%%%%%%%%%%%%%%%%%%%%%%%%%%%%%%%%%%%%%%%%%%%%%%%%%%%%%%%%
%%%%%%%%%%%%%%%%%%%%%%%%%%%%%%%%%%%%%%%%%%%%%%%%%%%%%%%%%%%%%%%%%%%%%

\section{$p$-Inner Functions}

From the theory of Hardy spaces, we say that $f \in H^2 \setminus \{0\}$ is an {\em inner function} if the radial boundary function for $f$ from \eqref{radial_BF} has constant modulus almost everywhere. In fact, most books on Hardy spaces make the normalizing assumption that this constant modulus is equal to one. Fourier analysis will show that $f \in H^2$ is inner precisely when 
$$\int_{0}^{\infty} |f(e^{i \theta})|^2 e^{-i n \theta} \frac{d \theta}{2 \pi} = 0, \quad n \in \N.$$ Since the inner product on $H^2$ can be given by 
$$\int_{0}^{2 \pi} f(e^{i \theta}) \overline{g(e^{i \theta})} \frac{d \theta}{2 \pi},$$ we see that $f \in H^2 \setminus \{0\}$ is an inner function when 
$$f \perp_{H^2} S^n f, \quad n  \in \N.$$ Furthermore \cite{Dur}, inner functions for $H^2$ take the form $f = c B S_{\mu}$, where $c \in \C \setminus \{0\}$, 
$$B(z) = z^m \prod_{k \geq 1} \frac{|w_k|}{w_k} \frac{w_k - z}{1 - \overline{w_k} z}$$ is a Blaschke product with zeros at $z = 0$, with multiplicity $m \in \N \cup \{0\}$, and $w_k \in \D \setminus \{0\}$ satisfying 
$$\sum_{k \geq 1} (1 - |w_k|) < \infty,$$ and 
$$S_{\mu}(z) = \exp\left(-\int_{0}^{2 \pi} \frac{e^{i \theta} + z}{e^{i \theta} - z} d \mu(\theta)\right),$$
where $\mu$ is a positive singular measure on $[0, 2 \pi]$.

The discussion above leads us to the following definition of inner for $\ell^{p}_{A}$, where we use Birkhoff-James orthogonality in place of Hilbert space orthogonality. 

\begin{Definition}\label{9847nvnsjdkh111}
   For $p \in (1, \infty)$ we say that $f \in \ell^p_{A} \setminus \{0\}$ is {\em $p$-inner} if $f \perp_{p} S^{n} f$ for all $n \in \N$. 
\end{Definition}

Observe from \eqref{BJp} that $f = \sum_{k \geq 0} a_k z^k \in \ell^{p}_{A}\setminus \{0\}$ is $p$-inner precisely when 
\begin{equation}\label{ooooodododo}
\sum_{k \geq 0} |a_{k + n}|^{p - 2} \overline{a_{k + n}} a_{k} = 0, \quad n \in \N.
\end{equation}
This shows that the function $f(z) = z^n$ is $p$-inner for all $n \in \N \cup \{0\}$. At first one might be at a loss to create other $p$-inner functions since the criteria in \eqref{ooooodododo} is difficult to apply and there is no product or integral representation like there was in the $p = 2$ case. However, the following analysis, using just the definition of Birkhoff-James orthogonality from \eqref{2837eiywufh[wpofjk}, provides many examples. 

\begin{Definition}
For $f \in \ell^{p}_{A}$, let $\widehat{f}$ be the metric projection of $f$ onto $[S f]$ from Definition \ref{24387tr9euoifjdkw}.
\end{Definition}

\begin{Proposition}\label{poiutorgorrrefdxxx}
For $f \in \ell^{p}_{A}$, let $J = f - \widehat{f}$ be the co-projection of $f$. Then $J$ is $p$-inner and any $p$-inner function arises in this way. 
\end{Proposition}

\begin{proof}
First observe that $\widehat{f}$ is the unique vector in $[S f]$ which satisfies 
$$\|f - \widehat{f}\|_{p} = \inf\|f - Q f\|_{p}: Q \in \mathcal{P}, Q(0) = 0\},$$
where $\mathcal{P}$ denotes the set of analytic polynomials. Then for any $t \in \C$ and $n \in \N$ we see that since $ t S^{n} (f - \widehat{f}) \in [S f]$ we have 
\begin{align*}
\|J\|_{p} & = \|f - \widehat{f}\|_{p}\\
& \leq \|f - \widehat{f} - t S^{n} (f - \widehat{f})\|_{p}\\
& = \|J - t S^{n} J\|_{p}.
\end{align*}
From the definition of Birkhoff-James orthogonality from \eqref{2837eiywufh[wpofjk}, $J \perp_{p} S^{n} J$ for all $n \in \N$ and thus $J$ is $p$-inner. 

Conversely, suppose that $J$ is $p$-inner. Then $J \perp_{p} S^n J$ for all $n \in \N$. But since the criterion for Birkhoff-James orthogonality in the $\ell^{p}_{A}$ setting \eqref{BJp} is linear in the second argument, we see that 
$J \perp_{p} Q J$ for all $Q \in \mathcal{P}, Q(0) = 0$. This means that 
\begin{align*}
\|J - \widehat{J}\| & = \inf\|J + Q J\|_{p}: Q \in \mathcal{P}, Q(0) = 0\}\\
& \geq \|J\|_{p}.
\end{align*}
By the definition of $\widehat{J}$, we see that $\widehat{J} = 0$ and so the $p$-inner function $J = J - \widehat{J}$ is of the desired form. 
\end{proof}

\begin{Proposition}
For any $f \in \ell^{p}_{A}$, the co-projection $J$ has at least the same zeros of $f$ with at least the same multiplicities. 
\end{Proposition}

\begin{proof}

Without loss of generality, assume $f(0)=1$. 
Indeed, if $f$ has a zero at the origin of multiplicity $m$, then we may carry out the following argument with $f/z^m$ in place of $f$.
 Let  $w$ be a zero of $f$ of multiplicity $n$.   We now argue that $w$ is a zero of $J$ with at least the same multiplicity.  
 
To see this, observe that by the definition of the co-projection, there must be polynomials $h_k$ such that $h_k(0) = 1$ and $h_k f$ converges to $J$ in norm.  This is due to the fact that  $\widehat{f}$ belongs to the subspace $S[f]$.  Since the difference quotient operator is a bounded operator on $\ell^p_A$ (Proposition \ref{jjjjdjdjsslpl}), it follows that 
$$\frac{h_k f}{(z - w)^n} \to \frac{J}{(z - w)^n}$$
in $\ell^p_A$ norm. In conclusion, $J/(z - w)^n$ is analytic on $\D$ and hence all zeros of $f$ must be zeros of $J$ to at least the same multiplicities.
\end{proof}

Let us work out the general formula for the $p$-inner function $J = f - \widehat{f}$ arising when $f$ is a polynomial with a finite number of zeros, all of which lie in $\D$.  Without serious loss of generality, we may assume that $f(0)=1$.

\begin{Example}

Let $p \in (1, \infty)$ and 
\[
  f(z) = 1 - \frac{z}{r}
\]
where $r \in \D \setminus \{0\}$.  Define $J = f-\widehat{f}$, so that $J \perp_p S^n f$ for all $n\in\mathbb{N}$. Since $\widehat{f} \in [S f]$ we have $\widehat{f}(0)  = 0$ and so $J(0)=1$. Moreover, if 
$$J(z) = \sum_{k \geq 0} J_k z^k = 1 + \sum_{k \geq 1} J_k z^k,$$
$$f(z) = \sum_{k \geq 0} f_k z^k = f_0 + f_1 z = 1 - \frac{1}{r} z,$$ we use \eqref{ooooodododo} to obtain
\[
    J_n^{\langle p-1 \rangle} f_0 + J_{n+1}^{\langle p-1 \rangle}f_1 = J_n^{\langle p-1 \rangle} \cdot 1 - J_{n+1}^{\langle p-1 \rangle}\cdot\frac{1}{r} = 0, \quad n \in \N.
\]
Solutions to this type of recurrence relation are well known;  for example
 \cite{Ela}. 
This has the obvious solution $J_n^{\langle p-1 \rangle} = Cr^n$.  By using the identity 
\begin{equation}\label{bbbcbcbcb112}
a^{\langle p-1 \rangle \langle p'-1 \rangle} =a
\end{equation} from Lemma \ref{776652bbbb}, we then have
\[
     J(z) = 1 + \sum_{k \geq 1} (C r^k)^{\langle p'-1 \rangle} z^k.
\]
The constant  $C$ is uniquely determined  by the requirement that $J(r) =0$.  This yields
\begin{align*}
     0 &= J(r) \\
        &= 1 + \sum_{k \geq 0} (C r^k)^{\langle p'-1 \rangle} r^k \\
        &= 1 + \sum_{k \geq 1} C^{\langle p'-1 \rangle} r^{p'k} \\
        &= 1 +  C^{\langle p'-1 \rangle}\frac{r^{p'}}{1-r^{p'}}.
\end{align*}
Thus 
       $$ C^{\langle p'-1 \rangle} = -\frac{1-r^{p'}}{r^{p'}}$$
       and so
\begin{align}
        J(z) &= 1 - \sum_{k \geq 1}  \frac{1-r^{p'}}{r^{p'}} r^{\langle p'-1 \rangle k} z^k  \nonumber \\
         &= 1 -  \frac{1-r^{p'}}{r^{p'}} \frac{r^{\langle p'-1 \rangle } z }{1 - r^{\langle p'-1 \rangle } z } \nonumber \\
         &= 1 -  \frac{1-r^{p'}}{r} \frac{z }{1 - r^{\langle p'-1 \rangle } z } \nonumber\\
         &= \frac{1 - r^{\langle p'-1 \rangle } z }{1 - r^{\langle p'-1 \rangle } z } - \frac{1-r^{p'}}{r} \frac{z }{1 - r^{\langle p'-1 \rangle } z } \nonumber\\
         &= \frac{1 - z/r}{1 - r^{\langle p'-1 \rangle } z }.\label{uuujujujuj}
\end{align}
This formula for the $p$-inner function associated with a linear polynomial was first derived in \cite{CR2} when exploring a possible ``inner-outer factorization'' for $\ell^p_A$.
\end{Example}

For a polynomial $f$ of higher degree, the following can be said about the $p$-inner function $f-\widehat{f}$.

\begin{Proposition}\label{uniqcoefthm}
Fix $p \in (1, \infty)$.  Suppose that $s_1$, $s_2$,\ldots, $s_d$ are distinct elements of $\mathbb{D} \setminus \{0\}$  and let $n_1$, $n_2$,\ldots, $n_d$ be positive integers.  Let $f$ be the polynomial
\[
     f(z) = \Big(1 - \frac{z}{s_1} \Big)^{n_1}\Big(1 - \frac{z}{s_2} \Big)^{n_2}\cdots\Big(1 - \frac{z}{s_d} \Big)^{n_d}.
\]
Then $J = f-\widehat{f} $ is of the form 
\begin{equation}\label{jgenpolyeq}
    J(z)  = 1 + \sum_{k=1}^{\infty} \big(\sum_{m=1}^{d} \sum_{j=0}^{n_m - 1} C_{j,m}k^j s_m^k  \big)^{\langle p'-1\rangle} z^k,
\end{equation}
and the constants $C_{j,m}$ are uniquely determined by the conditions that $J^{(m)}(s_k)=0$ for all $k$, $1 \leq k \leq d$ and all $m$, $0 \leq m \leq n_k-1$, where $J^{(m)}$ stands for the $m$th derivative of $J$.
\end{Proposition}

\begin{Remark}
Before proceeding to the proof of this proposition, it is worth remarking that $J$ is analytic in a neighborhood of $\overline{\D}$. Indeed, each Taylor coefficient for $J$ is a polynomial in $s_{m}$ raised to the $\langle p' - 1 \rangle$ power.  In that polynomial, one sees $k$ times some $s_{m}^k$.  This decays geometrically, since the modulus of $s_m$ is less than $1$.  In fact the root of largest modulus dictates the radius of convergence of $J$.  From here, one sees that the radius of convergence of the Taylor series defining $J$ is $1/R^{(p'-1)}>1$, where $R = \max\{|s_1|, |s_2|, \ldots, |s_m|\}$.
\end{Remark}

\begin{proof}[Proof of Proposition \ref{uniqcoefthm}]  By the definition of $J$ we have $J \perp_p S^k f$ for all $k \in \N$.  With $N = n_1+\cdots+n_d$, this gives rise to a recurrence relation
\begin{equation}\label{recreleq}
    J_{k}^{\langle p-1\rangle} f_0 +  J_{k+1}^{\langle p-1\rangle} f_1 + \cdots +
     J_{k+N}^{\langle p-1\rangle} f_{N} = 0, \quad k \in \N,
\end{equation}
on the coefficients of $J$, which, again via  \cite{Ela}, has the solution (\ref{jgenpolyeq}).

Next, suppose that $K$ is a function of the form (\ref{jgenpolyeq}), and it satisfies the condition $K^{(m)}(s_k)=0$ for all $k$, $1 \leq k \leq d$ and all $m$, $0 \leq m \leq n_k-1$. In other words, all of the roots of $f$ are zeros of $K$ to at least the same multiplicities. Then $K$ is analytic in a neighborhood of $\overline{\mathbb{D}}$, and its zero set contains those of $f$, multiplicities taken into account.  It follows that $K/f$ is also analytic in a neighborhood of  $\overline{\mathbb{D}}$. Consequently, we have
\[
     \phi_n f \to K
\]
in $\ell^p_A$, where $\phi_n$ is the $n$th partial sum of $K/f$.  This shows that $K \in [ f ]$.  Since $K(0) = f(0) = 1$, we also see that $f - K$ belongs to $S[ f ]$.  Finally, by virtue of $K$ having the form (\ref{jgenpolyeq}), the function $K$ satisfies 
\[
      K \perp_p S[ f ].
\]  
This forces $f - K$ to be the metric projection $\widehat{f}$ of $f$ onto $S[ f ]$.  We conclude that  $K  = J$.

Finally, we turn to proving the uniqueness of the constants $C_{j, m}$ defining $J$. Indeed, suppose that $J$ has two representations of the form (\ref{jgenpolyeq}), namely,
\begin{align*}
     J(z)  &= 1 + \sum_{k=1}^{\infty} \big(\sum_{m=1}^{d} \sum_{j=0}^{n_m - 1} C_{j,m}k^j s_m^k  \big)^{\langle p'-1\rangle} z^k \\
       &=  1 + \sum_{k=1}^{\infty} \big(\sum_{m=1}^{d} \sum_{j=0}^{n_m - 1} \widetilde{C}_{j,m}k^j s_m^k  \big)^{\langle p'-1\rangle} z^k.
\end{align*}
It must be that the corresponding coefficients coincide, or
\[
     \big(\sum_{m=1}^{d} \sum_{j=0}^{n_m - 1} C_{j,m}k^j s_m^k  \big)^{\langle p'-1\rangle}
     = \big(\sum_{m=1}^{d} \sum_{j=0}^{n_m - 1} \widetilde{C}_{j,m}k^j s_m^k  \big)^{\langle p'-1\rangle}
\]
for all $k \in \N$.  By the taking $\langle p-1 \rangle$th ``powers'' of both sides, using \eqref{bbbcbcbcb112} and transposing, we find that
\[
     \sum_{m=1}^{d} \sum_{j=0}^{n_m - 1} [C_{j,m}-\widetilde{C}_{j,m}] k^j s_m^k   = 0
\]
for all $k$.  The only way this can happen is if $C_{j,m} = \widetilde{C}_{j,m}$ for all $1 \leq m \leq d$ and $1\leq j \leq n_m-1$, since the sequences $(k^j s_m^k)_{k\geq 1}$ constitute a complete, linearly independent set of solutions to the difference equation (\ref{recreleq}) underlying $J$ 
 \cite[Corollary 2.24]{Ela}. This shows that the constants $C_{j,m}$ are uniquely determined by the conditions $J^{(m)}(s_k)=0$ for all $k$, $1 \leq k \leq d$ and all $m$, $0 \leq m \leq n_k-1$.
\end{proof}

%%%%%%%%%%%%%%%%%%%%%%%%%%%%%%%%%%%%%%%%%%%%%%%%%%%%%%%%%%%%%%%%%%%%%
%%%%%%%%%%%%%%%%%%%%%%%%%%%%%%%%%%%%%%%%%%%%%%%%%%%%%%%%%%%%%%%%%%%%%
%%%%%%%%%%%%%%%%%%%%%%%%%%%%%%%%%%%%%%%%%%%%%%%%%%%%%%%%%%%%%%%%%%%%%

 \section{Extremal Problems}

Using ideas from \cite{MR1398090}, we relate $p$-inner functions to certain extremal problems. 

For a sequence $W \subset \mathbb{D} \setminus\{0\}$, recall from \eqref{RW} that $\mathscr{R}_{W}$ denotes the functions in $\ell^{p}_{A}$ that vanish on $W$ with the appropriate multiplicities.  From here, consider the extremal problem 
\begin{equation}\label{extremeeq1}
    \sup\left\{ |f(0)| :  \|f\|_p =1, f \in \mathscr{R}_{W} \right\}
\end{equation}
This is equivalent to finding
\begin{equation}\label{extremeeq2}
    \inf \left\{ \|g\|_p :  g(0)=1, \ g \in \mathscr{R}_{W}\right\}.
\end{equation}
To see this, let $(f_n)_{n \geq 1}$ be a sequence for which approximates the extremum in \eqref{extremeeq1}. Then let
\[
    g_n(z) = \frac{f_n(z)}{f_n(0)}
\]
and notice that $g_n(0) =1$ and $g_n \in \mathscr{R}_{W}$.  Thus the infimum is no greater than the limit of the sequence $1/|f_n(0)|$.  On the other hand, suppose that the infimum is approximated by some sequence $(h_n)_{n \geq 1}$.  Let 
\[
     k_n(z) = \frac{h_n(z)}{\|h_n\|_p}.
\]
Thus $\|k_n\|_p =1$, and $k_n \in \mathscr{R}_{W}$.  Accordingly, the supremum is at least as large as the limit of the sequence $1/\|h_n\|_p$.

We have shown that the infimum is no greater than one over the supremum, and the supremum is at least one over the infimum.  This forces each to be the reciprocal of the other, and thus they are indeed equivalent problems in this sense.

The solutions to each extremal problem is unique, since the infimum is a nearest point to a closed convex subset of a uniformly convex space.

Let $H$ be the (unique) solution to the infimum problem (\ref{extremeeq2}).  Then by definition,
\[
    \|H\|_p \leq \| H(z) + z\Psi(z) \|_p
\]
for any  
$\Psi \in \mathscr{R}_{W}$. Conversely, this condition characterizes $H$.  
In particular, via \eqref{24387tr9euoifjdkw}, we have $H \perp_p S^n H$ for every $n\in\mathbb{N}$.

This $H$ reminds us of the $p$-inner functions $J = f - \widehat{f}$ that we have already encountered.  Let us check that they coincide in special cases.

\begin{Proposition}\label{lllklklklqlqlqlql}
   Let $p \in (1, \infty)$ and suppose that $W = (w_1, w_2,\ldots,w_n)$ is a finite sequence from $\mathbb{D}\setminus \{0\}$.  Let
 \[
        f(z) := \Big(1 - \frac{z}{w_1}  \Big)\Big(1 - \frac{z}{w_2}  \Big)\cdots \Big(1 - \frac{z}{w_n}  \Big)
  \]
and $J = f - \widehat{f}$ in $\ell^p_A$, and let $H$ be the solution to the infimum problem from  (\ref{extremeeq2}).  Then $J=H$.
\end{Proposition}

\begin{proof}  If $g\in \mathscr{R}_W$, then, by $n$ applications of Proposition \ref{jjjjdjdjsslpl}, the function $g/f$ also belongs to $\ell^p_A$.  Its power series converges in norm, hence 
there are polynomials $\phi_k$ such that $\phi_k \to g/f$ in $\ell^p_A$.  It follows that $\phi_k f \to g$.  This shows that the subspace of $\mathscr{R}_W$ is spanned by polynomial multiples of $f$.  With that, we may conclude
\begin{align*}
    \|H\|_p  &=  \inf\{\| g\|_p : g(0)=1, g|W=0\}\\
   &= \inf\{\|\phi f\|_p: \phi \in \mathcal{P}, \phi(0)=1\} \\
   & = \inf\{\|f + (\phi f - f)\|_{p}: \phi \in \mathcal{P}, \phi(0) = 1\}\\
   & = \inf\{\|f - Q f\|_{p}: Q \in \mathcal{P}, Q(0) = 0\}\\
  &= \|J\|_{p}. 
\end{align*}
This shows that $H = J$. 
\end{proof}

Let us examine another related extremal problem that will play an important role in the proof of our main zero set theorem (Theorem \ref{italiansubthm}). For a nonempty {\em finite} zero sequence $W \subset \D \setminus \{0\}$, recall that $\mathscr{R}_{W}$ denotes the functions in $\ell^{p}_{A}$ that vanish on $W$ with the appropriate multiplicities. Since $W$ is a finite and nonempty set, $\mathscr{R}_{W} \not = \{0\}$. 

Consider the extremal problem 
\begin{equation}\label{bbbbcbb1b1b1}
\inf\{\|1 + g\|_{p}: g \in \mathscr{R}_{W}\}.
\end{equation}

Observe that by the nearest point property for uniformly convex spaces, there is a unique $\Phi \in \ell^{p}_{A}$ for which $\Phi -1  \in \mathscr{R}_{W}$ and
\begin{equation}\label{ooo89we754}
\|\Phi\|_{p} = \inf\{\|1 + g\|_{p}: g \in \mathscr{R}_{W}\}.
\end{equation}

\begin{Proposition}\label{7734628234r}
The extremal problem in  \eqref{bbbbcbb1b1b1} has a unique solution $\Phi$ satisfying 
$$\Phi = 1 - \frac{H}{1 + (\|H\|_{p}^{p} - 1)^{p' - 1}},$$
where $H$ is the unique solution to the extremal problem in \eqref{extremeeq2}.
\end{Proposition}

\begin{Remark}\label{1515cvhgddf}  The proof of  Proposition \ref{7734628234r} relies on the following derivative calculation. 
Let $p \in (1, \infty)$ and $t$ be a real variable. Suppose that $h(t) = u(t) + iv(t)$, where $u$ and $v$ are differentiable functions of a real variable, not both vanishing on some interval.   
Then
\begin{align*}
     \frac{d}{dt} |h(t)|^p  &=  \frac{d}{dt}[ h(t)^{p/2}  \bar{h}(t)^{p/2}] \\
     &= \frac{ph(t)^{p/2} \bar{h}(t)^{p/2} \bar{h}'(t)}{2\bar{h}(t)} +  \frac{p\bar{h}(t)^{p/2} {h}(t)^{p/2} {h}'(t)}{2{h}(t)}\\
     &= p |h(t)|^{p-2} \Re[h'(t)\bar{h}(t)]\\
     &=  p \Re \big[ h(t)^{\langle p-1 \rangle} h'(t) \big].
\end{align*}
\end{Remark}

\begin{proof}[Proof of Proposition \ref{7734628234r}] First note that the above infimum measures the nearest point of $\mathscr{R}_{W}$ to the constant function $1$ and hence by the uniform convexity of $\ell^{p}_{A}$ it has a {\em unique} solution, i.e., there is a $G \in \mathscr{R}_{W}$ for which 
$$\inf\{\|1 + g\|_{p}: g \in \mathscr{R}_{W}\} = \|1 - G\|_{p}.$$
Next, observe that since $\mathscr{R}_{W} \not = \{0\}$, we have $\|1 - G\|_{p} \in (0, 1)$. Furthermore, if 
$$G(z) = \sum_{j \geq 0} G_{j} z^j$$ we can see from the identity 
$$\|1 - G\|_{p}^{p} = |1 - G(0)|^p + \sum_{j \geq 1} |G_j|^p,$$ that $G(0) \not = 0$.

In addition, we know that $G(0) > 0$. Indeed, if 
$$G_1(z) = G(z) \frac{|G(0)|}{G(0)},$$ then 
\begin{align*}
\|1 - G_1\|_{p}^{p} & = |1 - |G(0)||^p + \sum_{j \geq 1} |G_j|^p \\
& \leq |1 - G(0)|^p + \sum_{j \geq 1} |G_j|^p\\
& = \|1 - G\|_{p}^{p}.
\end{align*}
Since $G$ is the unique solution to the extremal problem in \eqref{bbbbcbb1b1b1}, it must be the case that $G_1 = G$ and so $G(0) = |G(0)| > 0$.

We now argue that $G(0) \in (0, 1)$. For this, note that $$\min\left\{\Big\|1 - \frac{t}{G(0)} G\Big\|_{p}^{p}: t \in \R\right\}$$
occurs when $t = G(0)$. On the other hand, 
$$\Big\|1 - \frac{t}{G(0)} G\Big\|_{p}^{p} = |1 - t|^p + |t|^p \sum_{j \geq 1} |G_j/G(0)|^p.$$
Taking derivatives with respect to $t$ on both sides of the previous line and using the calculation from Remark \ref{1515cvhgddf}, we get  
$$\frac{d}{d t}  \Big\|1 - \frac{t}{G(0)} G\Big\|_{p}^{p}= -p (1 - t)^{\langle p - 1\rangle} + p t^{\langle p - 1\rangle}   \sum_{j \geq 1} |G_j/G(0)|^p$$ which vanishes when 
$$(1 - t)^{\langle p - 1\rangle} = t^{\langle p - 1\rangle}   \sum_{j \geq 1} |G_j/G(0)|^p,$$
or equivalently when 
$$|1 - t|^{p - 2} (1 - t) = |t|^{p - 2} t \sum_{j \geq 1} |G_j/G(0)|^p.$$
The above can only happen when $t$ and $(1 - t)$ have the same sign -- which forces $t \in (0, 1)$. However, by the above analysis, this minimum occurs when $t = G(0)$, and so $G(0) \in (0, 1)$. 

Finally, among all the function $g = \sum_{j \geq 0} g_j z^j \in \mathscr{R}_{W}$ such that $g(0) = G(0)$, $G$ itself has the smallest norm. Indeed, just consider the identity 
$$\|1 + g\|_{p}^{p} = |1 + G(0)|^p + \sum_{j \geq 1} |g_j|^p.$$
Thus the function $G/G(0)$ satisfies the extremal problem in \eqref{extremeeq2}. Hence by Proposition \ref{lllklklklqlqlqlql}, $G/G(0) = H = J$. From here we see that 
$$\inf\{\|1 + g\|_{p}: g \in \mathscr{R}_{W}\} = \|1 - G\|_{p} = \|1 - G(0) H\|_{p}.$$ To compute the exact value of $G(0)$, observe that $G(0)$ is the value of $x \in (0, 1)$ for which 
$$\frac{d}{d x} \|1 - x H\|_{p}^{p} = 0.$$ By a computation similar to the one used above to show that $G(0) \in (0, 1)$ we see that 
\begin{align*}
\frac{d}{d x} \|1 - x H\|_{p}^{p} & = \frac{d}{d x}\left(|1 - x|^{p} + |x|^p (\|H\|_{p}^{p} - 1)\right)\\
& = -p (1 - x)^{\langle p - 1\rangle} + p x^{\langle p - 1\rangle} (\|H\|_p^{p} - 1)\\
& = -p (1 - x)^{p - 1} + p x^{p - 1} (\|H\|_p^{p} - 1).
\end{align*}
Note the use of the fact that $x$ and $1 - x$ are positive and Lemma \ref{776652bbbb} in the last step. Since $H(0) = 1$, observe that $\|H\|_{p}^{p} - 1 > 0$. 
One can check that 
$$-p (1 - x)^{p - 1} + p x^{p - 1} (\|H\|_p^{p} - 1) = 0$$ precisely when 
$$x = \frac{1}{1 + (\|H\|_{p}^{p} - 1)^{p' - 1}}.$$
This proves the desired formula for $\Phi = 1 - G$.
\end{proof}

\section{Proof of Theorem \ref{italiansubthm}}

Suppose that $W =  (w_1, w_2,\ldots) \subset \D \setminus \{0\}$. We will think of $W$ as a possible zero set, repeated to reflect multiplicities, for a function in $\ell^p_A$.  Set $W_n = (w_1, w_2,\ldots,w_n)$ and define $\mathscr{R}_{W_n}$ and $\mathscr{R}_{W}$ as in the previous section. 

Let $\Phi_n$ and $\Phi$ be the unique functions in $\ell^{p}_{A}$ satisfying the extremal problems
\[
       \|\Phi_n\|_p = \inf\{ \|1 + g\|_p:  g \in \mathscr{R}_{W_n}\},
\]
and 
\[
     \|\Phi\|_p = \inf\{ \|1 + g\|_p:  g \in \mathscr{R}_{W}\}
\]
from \eqref{ooo89we754}.  By Proposition \ref{7734628234r}, we know that
\[
     \Phi_n(z) = 1 - \frac{J_n(z)}{1 + (\|J_n\|^p_p-1)^{\langle p'-1 \rangle}}
\]

Now apply Lemma \ref{propnestlwplem}, to 
$\mathscr{X} = \ell^p_A$,  $\mathscr{X}_n = \{g \in \ell^p_A : g \in \mathscr{R}_{W_n} = 0 \}$ and $x=1$. The conclusion is that $\Phi_n$ (the metric co-projection of $1$ to $\mathscr{X}_{n}$) converges in norm to $\Phi$ (the metric co-projection of $1$ to $\mathscr{X}_{\infty}$). Furthermore, $\|\Phi\|_{p} < 1$ if and only if $\mathscr{R}_{W} \not = \{0\}$, i.e., $W$ is a zero set for $\ell^{p}_{A}$. 

We know that $\|\Phi_n\|_p$ is a nondecreasing sequence with $\|\Phi_{n}\|_{p} < 1$.  Moreover, 
since $J_n(0) = 1$ we have $\|J_n\|^{p}_{p} - 1 > 0$.
From Lemma \ref{776652bbbb} we get 
$$(\|J_n\|_{p}^{p} - 1)^{\langle p' - 1\rangle} = (\|J_n\|_{p}^{p} - 1)^{p' - 1}.$$
From here we have 
\begin{align*}
\|\Phi_n\|_{p}^{p} & = \|1 - G_n(0) J_n\|_{p}^{p}\\
& = |1 - G_n(0)|^p + |G_n(0)|^p (\|J_{n}\|_{p}^{p} - 1)\\
& = \left|1 - \frac{1}{1 + (\|J_n\|_{p}^{p} - 1)^{\langle p' - 1 \rangle}}\right|^p \\
& \qquad + \frac{1}{(1 + (\|J_n\|_{p}^{p} - 1)^{\langle p' - 1 \rangle})^p}(\|J_{n}\|_{p}^{p} - 1)\\
%& = \frac{(\|J_n\|_{p}^{p} - 1)^{\langle p' - 1 \rangle p}}{(1 + (\|J_n\|_{p}^{p} - 1)^{\langle p' - 1 \rangle})^p} + \frac{1}{(1 + (\|J_n\|_{p}^{p} - 1)^{\langle p' - 1 \rangle})^p}(\|J_{n}\|_{p}^{p} - 1)\\
%& =  \frac{( \|J_n\|^p_p -1 )^{p'}+( \|J_n\|^p_p -1 )}{(1+(\|J_n\|^p_p -1)^{p'-1})^p}\\
& = \frac{ \|J_n\|^p_p -1 }{(1+[\|J_n\|^p_p -1]^{p'-1})^{p-1}}.
\end{align*}
Take $p'-1$ powers of both sides of the above equation to get 
$$\|\Phi_n\|_{p}^{p'} = \frac{( \|J_n\|^p_p -1)^{p' - 1}}{1 + ( \|J_n\|^p_p -1)^{p' - 1}}.$$
Now solve for $( \|J_n\|^p_p -1)^{p' - 1}$ and then for $\|J_n\|_{p}^{p}$ to obtain
\[
      \|J_n\|^p_p = 1 + \frac{\|\Phi_n\|_p}{(1-  \|\Phi_n\|_p^{p'-1})^{p-1}}.
\]

Since $\|\Phi_n\|_p$ is a nondecreasing sequence with $\|\Phi_{n}\|_{p} < 1$, we see that $\|J_n\|_p$ must be  monotone nondecreasing with $n$. Moreover, $W$ is a zero set for $\ell^{p}_{A}$ if and only if $\|J_n\|_{p}$ is bounded.

 If $W$ is a zero sequence for $\ell^{p}_{A}$, the limiting function $J$ for the sequence $J_n$ is $p$-inner. To see this, note that  $J \in \ell^{p}_{A} \setminus \{0\}$ since $J_n(0) = 1$ for all $n$. Moreover, $J_{n} \perp_{p} S^{N} J_{n}$ for all $N \in \N$. Using the bilinear pairing $(\cdot, \cdot)$ between $\ell^{p}_{A}$ and its dual space $\ell^{p'}_{A}$, along with \eqref{ppsduwebxzrweq5}, we can rewrite this orthogonality condition as 
$$(J_{n}^{\langle p - 1\rangle}, S^{N} J_{n}), \quad N \in \N.$$

We now claim that $J_n \to J$ in $\ell^p_A$ if and only if $J^{\langle p-1\rangle}_n \to J^{\langle p-1\rangle}$ in $\ell^{p'}_A$.  To see this, let 
$a_k$ and $a_k^{(n)}$ be the $k$th coefficients of $J$ and $J_n$, respectively.  The hypothesis that $J_n \rightarrow J$ implies that 
$a_k^{(n)} \rightarrow a_k$ for each $k$; that is, viewed as functions of the index $k$, the coefficient sequence $J_k$ converges to $J$ ``pointwise.''  Furthermore, the elementary bound
\begin{align*}
    \Big| [a_k^{(n)}]^{\langle p-1 \rangle} -  a_k^{\langle p-1 \rangle}  \Big|^{p'} &\leq 2^{p'} \left(  
      \Big| [a_k^{(n)}]^{\langle p-1 \rangle}\Big|^{p'} + \Big|  a_k^{\langle p-1 \rangle}  \Big|^{p'}  \right) \\
      &= 2^{p'} \Big( | a_k^{(n)}|^{(p-1)p'} + | a_k|^{(p-1)p'}  \Big) \\
      &= 2^{p'} \Big( | a_k^{(n)}|^{p} + | a_k|^{p}  \Big)
\end{align*}
holds for all $k$ and $n$.  The right hand side is summable in $k$ for each $n$. Thus, it furnishes a suitable dominating sequence of functions of $k$, with the sequence being indexed by $n$, for the Dominated Convergence Theorem to apply.  Counting measure in $k$ is the underlying measure.  The conclusion is that  $J^{\langle p-1\rangle}_n \to J^{\langle p-1\rangle}$ in $\ell^{p'}_A$.  The converse holds since the argument is symmetric in  $p$ and $p'$.

As consequence of the claim, we see that 
$$(J^{\langle p - 1\rangle}, S^{N} J), \quad N \in \N.$$ In other words $J \perp_{p} S^{N} J$ for all $N$. This completes our proof.

\section{Geometric Convergence to the Boundary} \label{exampsec}

In this section we use our main theorem to construct  new examples of zero sequences for $\ell^{p}_{A}$.  To do this, we will find a bound on the norms of the co-projection functions  $J_n$ for each finite sequence $w_1, w_2,\ldots, w_n$.  Then, by applying Theorem \ref{italiansubthm}, we obtain a limiting function with the prescribed zeros.  The bound is made possible by constructing a polynomial that is suitably close each of the associated co-projections.

Let $p \in (1, \infty)$ and $W = (w_1, w_2, w_3,\ldots) \subset \D \setminus \{0\}$.  As in Theorem \ref{italiansubthm}, define, for each positive integer $n$,
\[
    f_n(z) :=  \Big( 1 - \frac{z}{w_1}  \Big)\Big( 1 - \frac{z}{w_2}  \Big)\cdots \Big( 1 - \frac{z}{w_n}  \Big),
\]
and the $p$-inner function  $J_n = f_n - \widehat{f}_n$.

Fix $\epsilon>0$ suitably small and select  $r_1, r_2, r_3,\ldots$ such that each $r_k >1$ and 
\[
       \sum_{k \geq 1} \Big(1 - \frac{1}{r_k}\Big) = \frac{1}{p'} - \epsilon > 0.
\]
Necessarily, $r_k \to 1$.  
For any $r \in (1, \infty)$ and $w\in \mathbb{D}\setminus \{0\}$, define 
\[
      B_{w,r}(z) :=  \frac{1-z/w}{1 - w^{\langle r'-1 \rangle}z},
\]
and observe this is the $r$-inner function from \eqref{uuujujujuj}. Also note that when $r = 2$, this function is a constant multiple of a Blaschke factor.

For each $n \in \N$, the function
\[
       F_n := B_{w_1,r_1} B_{w_2,r_2} \cdots B_{w_n,r_n}
\]
belongs to $\ell^p_A$, satisfies $F_n(0) =1$, and $F(w_k)=0$ for all $k$, $1 \leq k \leq n$. 
%
%     NOTE:   These functions utterly fail to play the role of FBP in l^p spaces.
%
 Thus, by the minimality property of the of co-projection $J_n$ (Definition \ref{24387tr9euoifjdkw}), we have 
\[
    \|J_n\|_p  \leq \|F_n\|_p
\]
for each $n$.  The goal now is to obtain an upper bound for $\|F_n\|_{p}$ that is independent of $n$. 

We now define  $p_1, p_2, p_3,\ldots$ by first defining 
\begin{equation}\label{useyoutheq}
     \frac{1}{p_1} + \frac{1}{r_1} = \frac{1}{p} + 1
\end{equation}
and then 
\begin{equation}\label{useyoutheq2}
      \frac{1}{p_k} + \frac{1}{r_k} = \frac{1}{p_{k-1}} + 1, \quad k \geq 2.
\end{equation}
Thus
\[
     \frac{1}{p_n} = \frac{1}{p} + \sum_{k=1}^{n}\Big( 1 - \frac{1}{r_k}  \Big),
\]
the sequence $(p_n)_{n \geq 1}$ is decreasing, and 
\[
     p^* := \lim_{n\rightarrow \infty} p_n  =  \left\{  \frac{1}{p} + \frac{1}{p'} -\epsilon  \right\}^{-1}  > 1.
\]

By virtue of the conditions (\ref{useyoutheq}) and (\ref{useyoutheq2}), we can apply Young's convolution inequality repeatedly to obtain
\begin{align}
     \| J_n \|_p  &\leq  \|F_n\|_p \nonumber \\
     &\leq \| B_{w_1,r_1} B_{w_2,r_2} \cdots B_{w_n,r_n} \|_p\nonumber\\
     &\leq \| B_{w_1,r_1} \|_{r_1}  \|B_{w_2,r_2} \cdots B_{w_n,r_n} \|_{p_1}\nonumber\\
     &\leq \| B_{w_1,r_1} \|_{r_1}  \|B_{w_2,r_2} \|_{r_2} \| B_{w_3,r_3}\cdots B_{w_n,r_n} \|_{p_2}\nonumber\\
     &\quad \cdots\nonumber\\
     &\leq \| B_{w_1,r_1} \|_{r_1}  \|B_{w_2,r_2} \|_{r_2} \| B_{w_3,r_3}\|_{r_3} \cdots \|B_{w_{n-1},r_{n-1}} \|_{r_{n-1}}\|B_{w_{n},r_{n}} \|_{p_{n-1}}
               \nonumber\\
     &\leq \| B_{w_1,r_1} \|_{r_1}  \|B_{w_2,r_2} \|_{r_2} \| B_{w_3,r_3}\|_{r_3} \cdots \|B_{w_{n-1},r_{n-1}} \|_{r_{n-1}}\|B_{w_{n},r_{n}} \|_{p^*}.
     \label{laststepinyoungeq}
\end{align}
We will be done if we can find a uniform bound for the final factor, $\|B_{w_{n},r_{n}} \|_{p^*}$, as well as the product of the remaining factors.

By direction calculation, similar to \cite[Lemma 3.2]{MR3686895}, we have
\[
     \|B_{w,r} \|_{t}^t  =  1 + \frac{(1 - |w|^{r'})^t}{|w|^t (1 - |w|^{(r' -1)t})}.
\]
and 
\[
     \|B_{w,r} \|_{r}^r  =  1 + \frac{(1 - |w|^{r'})^{r-1}}{|w|^r }.
\]
Let us prove the second identity since the proof of the first identity is similar. Indeed, 
\begin{align*}
  \| B_{w, r} \|^{r}_{r}  &=  1 + \sum_{j \geq 1} \Big|  {w^{\langle r'-1 \rangle (j -1)}} \Big(w^{\langle r'-1 \rangle} - \frac{1}{w}\Big) \Big|^r  \\
&=  1 + \sum_{j \geq 1} \Big|  {w^{\langle r'-1 \rangle (j -1)}} \Big(w^{\langle r'-1 \rangle} - \frac{1}{w}\Big) \Big|^r  \\
&=  1 + \Big|  w^{\langle r'-1 \rangle} - \frac{1}{w}\Big|^r \sum_{j=1}^{\infty} |w|^{r(r'-1)}   \\
&=  1 + \frac{(1-|w|^{r'})^{p}}{|w|^r}  \frac{1}{1-|w|^{r'}} \\
&=  1 + \frac{(1-|w|^{r'})^{r-1}}{|w|^p}.
\end{align*}

When $t \in (1, \infty)$, $r \in (1, \infty)$, and $|w|$ increases to 1, the quantity $\|B_{w,r} \|_{t}^t$ tends to the value 1, since  
\begin{equation}
      \frac{(1 - |w|^{r'})^t}{(1 - |w|^{(r' -1)t})} \sim  \frac{t(1-|w|^{r'})^{t-1}r'|w|^{r'-1}}{t(r'-1)|w|^{t(r'-1)-1}} \to 0. 
\end{equation}

Next, we recall that $p^* >1$, and $r_k$ decreases to 1 as $k$ tends to infinity.  Thus $r'_k$ increases to infinity, and consequently,
for $k$ sufficiently large, we have
\begin{align*}
   p^* &\geq r'_k /(r'_k -1)\\
   (r'_k -1)p^* &\geq r'_k \\
   1-|w_k|^{(r'_k -1)p^*} &\geq 1- |w_k|^{r'_k} \\
   &\geq (1- |w_k|^{r'_k})^{p^*}.
\end{align*}
This implies that for sufficiently large $k$, the last factor $\|B_{w_{n},r_{n}} \|_{p^*}$ of (\ref{laststepinyoungeq}) is no greater than $2$.
Finally, we see that (\ref{laststepinyoungeq}) is uniformly bounded as $n$ increases provided that the roots $w_1, w_2, w_3,\ldots$ satisfy
\begin{equation}\label{77wejsdhf}
     \sum_{k \geq 1}  (1 - |w_k|^{r'_{k}})^{r_k-1} < \infty.
\end{equation}

This proves Theorem \ref{blaslikeex}, which asserts that a sequence
  $W = (w_j)_{j \geq 1}$ satisfying  \eqref{77wejsdhf} is a zero set for $\ell^p_A$.

It is obvious that any $W$ satisfying (\ref{77wejsdhf}) must be a Blaschke sequence.  However, by replacing each factor
$B_{w_k,r_k}(z)$ in the above construction with $B_{w_k,r_k}(z^k)$, we obtain a zero set $\widetilde{W}$ consisting of the complex $k$th roots of $w_k$ for each $k$.  Because
\[
     \|B_{w_k,r_k}(z)\|_{r_k} = \|B_{w_k,r_k}(z^k)\|_{r_k}
\] 
holds, the same estimate for (\ref{laststepinyoungeq}) applies, telling us that $\widetilde{W}$ a zero set for $\ell^p_A$.   It is clear that $\widetilde{W}$ cannot be the union of finitely many sequences tending toward the boundary at exponential rates.  
Furthermore, the zero set $\widetilde{W}$ accumulates everywhere on the boundary of $\mathbb{D}$.  Therefore, we have produced an example going beyond those known from the zero sets created in \cite{MR0148874,MR947146}, via Blaschke products having certain desired properties on their Taylor coefficients.

\section{Slower Than Geometric Convergence to the Boundary}

First, let us observe that if $r_k$ is given by
\[
     r_k := e^{-1/k},  \, \ k \in \mathbb{N},
\]
then $r_k$ fails to converge to 1 {\em at an exponential rate} as $k$ increases to infinity. 
In fact,
\[
     \lim_{k\rightarrow\infty} \frac{1 - r_{k+1}}{1-r_k} =
     \lim_{k\rightarrow\infty}\frac{1 - e^{-1/(k+1)}}{1 - e^{-1/k}} % = \lim_{k\rightarrow\infty}\frac{ e^{-1/(k+1)}/(k+1)^2  }{  e^{-1/k}/k^2 }
     = \lim_{k\rightarrow\infty}\Big(  \frac{k}{k+1} \Big)^2 e^{1/k(k+1)}
     = 1.
\]
The conclusion remains valid if $e$ is replaced by some other base exceeding one, or if the $1/k$ in the exponent is replaced by $1/k^d$ for any positive integer $d$. Furthermore, it holds all the more if, instead of being constant, the base increases with $k$.  This will come into play at the end of our construction.

Our overall strategy in our construction is to identify an increasing sequence of nested finite zero sets and define $f_k$ to to be the polynomial with precisely the zeros of the $k$th set.  We will obtain a corresponding sequence of polynomials $F_k$ that carry these zero sets and other zeros as well.  Each $F_k$ will furnish a norm estimate of the associated $p$-inner co-projection function $J_k = f_{k} - \widehat{f_k}$.  By showing that these $F_k$ are uniformly bounded in norm, and using the extremal property of $J_k$, i.e., $\|J_k\|_{p} \leq \|F_{k}\|_{p}$, we may conclude that the $J_k$ are also bounded, and hence the constructed zero set is contained in that of a nontrivial function in $\ell^p_A$.  

Let $p \in (1, \infty)$ and for each $k \in \N$ consider polynomials $F_k$ given by
\begin{align}
    F_k(z) &:=  \left(  1 - \frac{z}{r_1}  \right) \left( 1 - \frac{1}{2} \Big[ \frac{z^2}{r_2^2} + \frac{z^4}{r_2^4}  \Big]  \right)
         \left( 1 - \frac{1}{4} \Big[ \frac{z^8}{r_3^8} + \frac{z^{16}}{r_3^{16}} +\frac{z^{32}}{r_3^{32}} + \frac{z^{64}}{r_3^{64}}  \Big]  \right)\nonumber\\
         &\qquad \times \cdots \times  
          \left( 1 - \frac{1}{2^{k-1}} \Big[ \frac{z^{N_k}}{r_k^{N_k}} + \cdots + \frac{z^{N_k^2}}{r_k^{N_k^2}}  \Big]  \right)\label{formfsubkeq},
\end{align}
where $N_k = 2^{(2^{k-1}-1)}$, and $r_1, r_2, \ldots, r_k$ belong to $(0,1)$ yet to be determined.
Let us make some observations about these $F_k$.  Each factor is a polynomial with a number of roots; among them are a specific set of roots that we will call the {\em targeted roots}.  The targeted roots are determined in the following way.  For each  $j = 1, 2, 3,\ldots, k$, fix some modulus $r_j$ with $r_j \in (0, 1)$ and notice that the $j$th factor vanishes, as does $F_k$ itself, consequently, at the points
\[
     r_j,\  r_j e^{2\pi i/N_j},\ r_j e^{2\cdot 2\pi i/N_j}, \ldots,\ r_j e^{(N_j -1)\cdot 2\pi i/N_j}.
\]
For each $j$, the targeted roots are these $N_j$ elements of $\mathbb{D}$, each with modulus $r_j$, uniformly distributed in argument around the disk.  The $j$th factor thus contributes a huge number of targeted roots, and this serves to slow down the rate of convergence to the boundary of the constructed zero set.  Any roots other than the targeted roots will have no bearing whatsoever on the argument. Define $f_{k}$ to be the polynomial whose zeros are precisely the targeted roots of $F_k$. 

Next, consider the effect of multiplying out the factors $F_k$, with intention of calculating its norm in $\ell^p_A$.  The fact that all occurrences of $z$ in the defining formula for $F_k$ are all powers of $2$ implies that in the expansion each $z^m$ can occur only once for each $m$ (namely, the combination of factors corresponding to the binary representation of $m$).  Put differently, when you multiply out the defining formula for $F_k$, there is no need to collect like terms -- each power of $z$ can only arise in at most one way.  

A typical term in this expansion looks like
\[
    \pm \frac{1}{2^{j_1-1} 2^{j_2-1} \cdots 2^{j_s-1}} \frac{1}{r_{j_1}^{m_1} r_{j_2}^{m_2} \cdots r_{j_s}^{m_s}} z^m,
\]
where $m_1$, $m_2$,\ldots,$m_s$ are certain powers of 2 adding up to $m$.  Its absolute value bounded above crudely by
\[
     \frac{1}{2^{j_1-1} 2^{j_2-1} \cdots 2^{j_s-1}} \frac{1}{r_{j_1}^{N_{j_1}^2} r_{j_2}^{N_{j_2}^2} \cdots r_{j_s}^{N_{j_s}^2}},
\]
and notice that there are $2^{j_1-1} 2^{j_2-1} \cdots 2^{j_s-1}$ terms with the same bound.  It follows that $\|F_k\|^p_p$ is bounded above by
a sum of terms of the form
\[
      \frac{1}{(2^{j_1-1} 2^{j_2-1} \cdots 2^{j_s-1})^{p-1}} \frac{1}{r_{j_1}^{N_{j_1}^2 p} r_{j_2}^{N_{j_2}^2 p} \cdots r_{j_s}^{N_{j_s}^2 p}}.
\] 
where the parameters $j_1$, $j_2$,\ldots,$j_s$ are now distinct.

And now working backwards from this sum, we obtain the bound
\begin{align*}
    \|F_k\|^p_p & \leq \left(1 + \frac{1}{r_1^p}  \right)\left(1 + \frac{1}{2^{p-1}r_2^{4p}} \right)\left(1 + \frac{1}{2^{2(p-1)}r_3^{64p}} \right)\\
       & \qquad \times \cdots \times \left(1 + \frac{1}{2^{(k-1)(p-1)}r_k^{N_k^2 p}} \right).
\end{align*}

The infinite product converges if and only if the following sum converges:
\[
      \sum_{k=1}^{\infty} \frac{1}{2^{(k-1)(p-1)}r_k^{N_k^2 p}}.
\]

Our next task will be to identify values of $r_k$ that are sufficient for this sum to converge. Plainly, this happens if there is some $a>1$ such that
\[
       \frac{1}{2^{(k-1)(p-1)}r_k^{N_k^2 p}}  = \frac{1}{k (\log k)^{a}}
\]
for all $k$.  This can be rewritten as
\[
       r_k = \left( \frac{k (\log k)^{a}}{2^{(k-1)(p-1)} }  \right)^{1/2^{(2^k -2)p}}.
\]

Now remember that $r_k$ is not the modulus of the $k$th of the zeros of $F_k$, but rather, it is the modulus of a large collection of zeros.
Indeed, if $(\rho_n)_{n\geq 1}$ is an enumeration of the set of targeted zeros in order of non-decreasing modulus, then $|\rho_n| = r_k$
whenever
\[
 N_1 + N_2 + \cdots + N_{k-1} < n \leq N_1 + N_2 + \cdots + N_{k-1} + N_k.
\]

Let us estimate the rate at which the targeted roots of $F_k$ tend toward the boundary. 
Since $N_{k-1} \leq N_1 + N_2 + \cdots + N_{k-1}$ and $N_1 + N_2 + \cdots + N_{k-1} + N_k \leq k N_k$ for all $k$, we have the bounds
$ N_{k-1}  < n  \leq k N_k$.
With $n$ and $k$ related in this manner, it follows that
\begin{align*}
N_{k-1} &< n \leq k N_k \\
%2^{(2^{k-2}-1)} &< n \leq k \cdot 2^{(2^{k-1}-1)}\\
%2^{(2^{k-2}-1)} &< n \leq 2^{(2^{k-1}-1)}\cdot 2^{(2^{k-1}-1)}\\
%2^{(2^{k-2}-1)} &< n \leq \ 2^{(2^{k}-1)}\\
%{(2^{k-2}-1)} &< \log_2 n \leq {(2^{k}-1)}\\
%k-2 &< \log_2(1 + \log_2 n) \leq k \\
\log_2(1 + \log_2 n) \leq k &< 2 + \log_2(1 + \log_2 n).
\end{align*}
Consequently 
\begin{align*}
\rho_n & \geq \left(\frac{ [\log_2(1 + \log_2 n)] \{\log_2(1 + \log_2 n)\}^a}{\{2[1+\log_2 n]\}^{p-1} } \right)^{(2/n)^p} \\
\rho_n & \leq \left(\frac{ [2 + \log_2(1 + \log_2 n)] \{2 + \log_2(1 + \log_2 n)\}^a}{\{(1/2)[1+\log_2 n]\}^{p-1} } \right)^{1/(4n^4)^p}.
\end{align*}
By the observation made at the beginning of this section, the zero sequence $(\rho_n)_{n\geq 1}$ fails to approach the boundary at a geometric rate.

This completes the construction.  Again, this produces an example of a zero set that fails to satisfy the Newman condition; that is, it cannot be expressed as the union of sequences tending toward the boundary of $\mathbb{D}$ at an exponential rate.

\section{A Non-Blaschke Zero Set for $p>2$}

Vinogradov \cite{Vinogradov} proved  that when $p > 2$ the zero sets for $\ell^{p}_{A}$ need not be Blaschke sequences. Here is a new proof of this using $p$-inner functions. For each $k \in \N$, define the polynomial $F_k$ by
\begin{align*}
   F_k(z) &:= \left( 1 - \frac{z}{r_1} \right) \left( 1 - \frac{1}{2}\Big[\frac{z^{2!}}{r_2^{2!}} + \frac{z^{2\cdot 2!}}{r_2^{2\cdot 2!}} \Big] \right)
                \left( 1 - \frac{1}{3}\Big[\frac{z^{3!}}{r_3^{3!}} + \frac{z^{2\cdot 3!}}{r_3^{2\cdot 3!}} + \frac{z^{3\cdot 3!}}{r_3^{3\cdot 3!}}\Big] \right)\\
                &\qquad \times\cdots\times
         \left( 1 - \frac{1}{k}\Big[\frac{z^{k!}}{r_k^{k!}} + \frac{z^{2\cdot k!}}{r_k^{2\cdot k!}} + \frac{z^{3\cdot k!}}{r_k^{3\cdot k!}}
               +\cdots+   \frac{z^{k\cdot k!}}{r_k^{k\cdot k!}}  \Big] \right),
\end{align*}
where $r_1, r_2, r_3,\ldots$ are moduli in $(0,1)$.  We observe that among the roots of this polynomial are certain {\it targeted roots}, consisting of
$$r_j, r_j e^{2\pi i/j!}, r_j e^{2\pi i \cdot 2/j!}, \ldots, r_j e^{2\pi i \cdot (j!-1)/j!}, \quad 1 \leq j \leq k.$$  Observe how the $j$th factor contributes $j!$ roots. Again, as with the example in the previous section, let $f_k$ be the polynomial whose roots are {\em precisely} the targeted roots of $F_k$. 

It is easily proved by induction that
\[
      1 + 2\cdot 2! + 3\cdot 3! + \cdots + k\cdot k!  + 1 = (k+1)!
\]
for each $k$.  
As a consequence, when the formula for $F_k$ is multiplied out, each resulting power of $z$ can only arise from one combination of factors, and there is no need to collect like terms.  (It helps to notice that if we take the largest power of $z$ from each of the first $(k-1)$ factors and multiply them, then the resulting power is, by design, one less than the smallest power of the $k$th factor.)

This greatly simplifies the estimate of $\|F_k\|_p$.  Indeed, a typical term in the expansion looks like
\[ \pm 
 \frac{1}{k_1 k_2 \cdots k_m} \frac{1}{r_{k_1}^{n_1 k_1 !}r_{k_2}^{n_2 k_2 !}\cdots r_{k_m}^{n_m k_m !}}z^{(n_1 k_1!+ n_2 k_2! + n_m k_m!)}
\]
where $k_1$, $k_2$,\ldots, $k_m$ are distinct indices between $1$ and $k$, and for $1\leq l \leq m$, we have $1 \leq n_l \leq k_l$.  The coefficient can be bounded absolutely above by
\[
      \frac{1}{k_1 k_2 \cdots k_m} \frac{1}{r_{k_1}^{k_1 k_1 !}r_{k_2}^{k_2 k_2 !}\cdots r_{k_m}^{k_m k_m !}}
\]
where we have simply replaced powers in the denominator by something possibly larger, increasing the fraction overall.  There are exactly
$k_1 k_2 \cdots k_m$ terms in the expansion with the same bound. Accordingly we obtain the estimate

\[
    \|F_k\|_p^p  \leq 1 + \sum \frac{1}{(k_1 k_2 \cdots k_m)^{p-1}} \frac{1}{r_{k_1}^{pk_1 k_1 !}r_{k_2}^{pk_2 k_2 !}\cdots r_{k_m}^{pk_m k_m !}}
\]

where the sum ranges over all selections $k_1$, $k_2$,\ldots, $k_m$ of distinct indices between $1$ and $k$.   The right hand side can be expressed as

\[
     \left( 1 + \frac{1}{1^{p-1} r_1^{p}}  \right)\left( 1 + \frac{1}{2^{p-1} r_2^{2p\cdot 2!}}  \right)\cdots \left( 1 + \frac{1}{k^{p-1} r_k^{kp\cdot k!}}  \right)
\]
which converges as $k$ tends to infinity precisely when
\[
      \sum_{k=1}^{\infty}  \frac{1}{k^{p-1} r_k^{kp\cdot k!}} \ < \ \infty
\]
Convergence is assured if we take, for example,
$$r_k  =  \Big(\frac{1}{k^{p - 2 -\alpha}}\Big)^{1/kp\cdot k!},$$
where $\alpha > 0$.  This can only make sense if $p >2$, and we choose $0 < \alpha < p-2$.   The defined sequence of polynomials therefore satisfies
\[
      \sup_{k\geq 1} \| F_k \|_p  < \infty.
\]
Since each $F_k(0) = 1$, we have $ \| J_k \|_p \leq \|F_k \|_p$ for all $k$, where $J_k$ is the $p$-inner function corresponding to $f_k$, the polynomial with precisely the targeted roots up to the $k$th index.  Now invoke  Theorem \ref{italiansubthm} to see that the set $W$ of all targeted zeros is the zero set of a nontrivial function from $\ell^p_A$.

In this case, the corresponding Blaschke test (recall \eqref{024ouriehtrjgeqw}) becomes 
\begin{align*}
      \sum_{k=1}^{\infty}  k! ( 1 - r_k) &= \sum_{k=1}^{\infty}  k! ( 1 - \exp\log r_k) \\
       &=  \sum_{k=1}^{\infty}  k! \left( 1 - \exp \left\{ \Big[\frac{1}{kp\cdot k!
      }\Big ] (\log \{1/k\})(p - 2 -\alpha) \right\} \right) \\
     & \geq  \  \sum_{k=1}^{\infty} k! \frac{1}{kp\cdot k!}( \log k)(p-2-\alpha)\\
     & \qquad -   
     \sum_{k=1}^{\infty} k! \frac{1}{2(kp\cdot k!)^2}( \log k)^2(p-2-\alpha)^2\\
     &= \sum_{k=1}^{\infty} \frac{\log k}{kp}(p-2-\alpha) -   
     \sum_{k=1}^{\infty}  \frac{1}{2(kp)^2k!}( \log k)^2(p-2-\alpha)^2
\end{align*}

(recalling that there are $k!$ roots with modulus $r_k$)
which diverges to infinity.   Here we used 

\[
      1 - e^{-x} = 1 - \left(1 - x + \frac{x^2}{2!} - \frac{x^3}{3!} + \cdots\right)  \geq 1 - \left(1-x+ \frac{x^2}{2}\right) = x - \frac{x^2}{2}.
\]
for sufficiently small $x$.

\bibliographystyle{plain}

\bibliography{references4}

%%%%%%%%%%%%%%%%%%%%%%%%%%%%%%%%%%%%%%%%%%%%%%%%%%%%%%%%%%%%%%%%%%%%%%%
%
%                               DOCUMENT ENDS HERE
%
%%%%%%%%%%%%%%%%%%%%%%%%%%%%%%%%%%%%%%%%%%%%%%%%%%%%%%%%%%%%%%%%%%%%%%
%
%
%
\end{document}